\documentclass[11pt]{article}
\usepackage{color,latexsym,amsthm,amsmath,amssymb,path,enumerate,url}
\usepackage{verbatim}
\newcommand{\ignore}[1]{}

\newtheorem{theorem}{Theorem}[section]
\newtheorem{lemma}[theorem]{Lemma}

\newcommand{\Proof}[1]
        {
        \noindent
        \emph{Proof #1.}~
        }
\newsavebox{\smallProofsym}                     
\savebox{\smallProofsym}                        %
        {
        \begin{picture}(7,7)                    %
        \put(0,0){\framebox(6,6){}}             %
        \put(0,2){\framebox(4,4){}}             %
        \end{picture}                           %
        }                                       %
\newcommand{\smalleop}[1]
        {
        \mbox{} \hfill #1~~\usebox{\smallProofsym}\!\!\!\!\!\!\
        }

\newcommand{\parag}[1]{\vspace{2mm}

\noindent{\bf #1} }

\usepackage{graphicx,psfrag}
\usepackage[rflt]{floatflt}

\usepackage{dsfont}

\newcommand{\ZZ}{\ensuremath{\mathbb Z}}
\newcommand{\RR}{\ensuremath{\mathbb R}}
\newcommand{\CC}{\ensuremath{\mathbb C}}

\newcommand{\pts}{\mathcal P}
\newcommand{\vars}{\mathcal V}
\newcommand{\curves}{\mathcal C}
\newcommand{\flats}{\mathcal L}
\newcommand{\lines}{\mathcal L}

\newcommand{\hyperp}{\mathcal H}

\newcommand{\rank}{\mathrm{rank\,}}

\newcommand{\lattice}{\mathcal L}
\newcommand{\vb}{{\bf V}}
\newcommand{\ib}{{\bf I}}
\newcommand{\jb}{{\bf J}}

\def\eps{{\varepsilon}}

\begin{document}
\pagenumbering{arabic}

\title{A General Incidence Bound in $\RR^d$ and Related Problems}

\author{
Thao Do\thanks{Department of Mathematics, Massachusetts Institute of Technology, Cambridge, MA, USA.
{\sl thaodo@mit.edu}.}
\and
Adam Sheffer\thanks{Department of Mathematics, Baruch College, City University of New York, NY, USA.
{\sl adamsh@gmail.com}. Supported by NSF grant DMS-1710305.}}

\maketitle

\begin{abstract}
We derive a general upper bound for the number of incidences with $k$-dimensional varieties in $\RR^d$.
The leading term of this new bound generalizes previous bounds for the special cases of $k=1, k=d-1,$ and $k= d/2$, to every $1\le k <d$.
We derive lower bounds showing that this leading term is tight in various cases.
We derive a bound for incidences with transverse varieties, generalizing a result of Solymosi and Tao.
Finally, we derive a bound for incidences with hyperplanes in $\CC^d$, which is also tight in some cases.
(In both $\RR^d$ and $\CC^d$, the bounds are tight up to sub-polynomial factors.)

To prove our incidence bounds, we define the \emph{dimension ratio} of an incidence problem.
This ratio provides an intuitive approach for deriving incidence bounds and isolating the main difficulties in each proof.
We rely on the dimension ratio both in $\RR^d$ and in $\CC^d$, and also in some of our lower bounds.

\end{abstract}

\section{Introduction}

Geometric incidence is an important topic in Discrete Geometry.
Given a set $\pts$ of points and a set $\vars$ of geometric objects (such as circles or hyperplanes) in  $\RR^d$, an \emph{incidence} is a pair $(p,V)\in \pts \times \vars$ such that the point $p$ is contained in the object $V$.
The number of incidences in $\pts\times\vars$ is denoted as $I(\pts,\vars)$.
In incidence problems, one is usually interested in the maximum number of incidences in $\pts \times \vars$, taken over all possible sets $\pts,\vars$ of given sizes.
Such incidence bounds have many applications in a variety of fields.
For a few recent examples, see Guth and Katz's solution to the Erd\H os distinct distances problem \cite{GK15}, a number theoretic result by Bombieri and Bourgain \cite{Bb15}, and works in Harmonic Analysis such as \cite{BD15,KZ17}.

When studying an incidence problem between a point set $\pts$ and a set of objects $\vars$, we sometimes consider the \emph{incidence graph} of $\pts \times \vars$.
This bipartite graph has vertex sets $\pts$ and $\vars$, and an edge for every incidence.
Deriving an upper bound for the number of incidences is equivalent to finding an upper bound for the number of edges in the incidence graph.
When studying incidence problems in dimension $d\ge 3$, one usually assumes that the incidence graph contains no copy of $K_{s,t}$ for some constants $s,t\ge 2$.
Such incidence problems can also be thought of as algebraic or geometric variants of the Zarankiewicz problem (for example, see \cite{FPSSZ17}).

In this paper we study incidences with varieties of any dimension in $\RR^d$, when the incidence graph contains no copy of  $K_{s,t}$ for some constants $s,t\ge 2$.
The following theorem describes the main known results that hold for every $d\ge 2$.
We use the notation $f=O_{a_1,\dots, a_k}(g)$ to indicate there is some positive constant $c$ that depends on $a_1\dots, a_k$, such that $f\leq c\cdot g$.
For a point set $\pts$ and a set $\vars$ of varieties, both in $\RR^d$, we denote by $I^*(\pts,\vars)$ the number of incidences $(p,h)\in \pts\times \vars$ where $p$ is a regular point of $h$.

\begin{theorem} \label{th:UpperBounds}
Let $\pts$ be a set of $m$ points and let $\vars$ be a set of $n$ varieties of degree at most $D$, both in $\RR^d$, such that the incidence graph of $\pts\times \vars$ contains no copy of $K_{s,t}$.

(a) (Solymosi and Tao \cite{ST12}) Assume that every variety of $\vars$ is of dimension at most $d/2$ and that the varieties intersect \emph{transversely} (that is, whenever two varieties intersect in a point $p$, their tangent spaces at $p$ intersect in a single point).
Then for every $\eps>0$ we have \\
\[I^*(\pts,\vars) = O_{D,s,t,d,\eps}\left(m^{\frac{s}{2s-1}+\eps}n^{\frac{2s-2}{2s-1}}+m+n \right). \]

(b) (Fox, Pach, Sheffer, Suk, and Zahl \cite{FPSSZ17}) For any $\eps>0$ we have
\[I(\pts,\vars) = O_{D,s,t,d,\eps}\left(m^{\frac{(d-1)s}{ds-1}+\eps}n^{\frac{d(s-1)}{ds-1}}+m+n \right). \]

(c) (Sharir, Sheffer, and Solomon \cite{SSS16}) When every variety of $\vars$ is of dimension at most one, for every $\eps>0$ there exists a constant $c$ that satisfies the following.
For $j=2,\ldots,d-1$, assume that every $j$-dimensional variety of degree at most $D$ contains at most $q_j$ varieties of $\vars$, for parameters $q_2 \le \cdots \le q_{d-1}\le q_d=n$.
Moreover, for every $2\le j<l \le d$, we have $q_j \ge \left(\frac {q_{l-1}} {q_l}\right)^{l(l-2)}q_{l-1}$.
Then for any $\eps>0$ we have
\begin{align*}
I(\pts,\vars) = O_{D,s,t,d,\eps}\Bigg(m^{\frac{s}{ds-d+1}+\eps}n^{\frac{ds-d}{ds-d+1}}+ \sum_{j=2}^{d-1}
m^{\frac{s}{js-j+1}+\eps}n^{\frac{d(j-1)(s-1)}{(d-1)(js-j+1)}}&
q_j^{\frac{(d-j)(s-1)}{(d-1)(js-j+1)}} \\
&\hspace{15mm}+m+n\Bigg).
\end{align*}
\end{theorem}
When $2\le d \le 4$, additional results are known. For example, see \cite{BS16,SZ17,Zahl17}.

The bound of part (a) of Theorem \ref{th:UpperBounds} is considered ``good'' only for varieties of dimension exactly $d/2$.
For the exact meaning of good, see the discussion below.
For now we only state that this bound is known to be tight up to sub-polynomial factors in some cases, but only when the varieties are of dimension exactly $d/2$.
The bound of part (b) is considered ``good'' when the varieties are of dimension $d-1$ (and is tight up to sub-polynomial factors in some cases for varieties of dimension $d-1$).
The bound of part (c) holds only when the varieties are of dimension one.

All of the bounds in Theorem \ref{th:UpperBounds} are obtained using the \emph{polynomial partitioning} technique (see section \ref{sec:prelim}).
When using this technique for $k$-dimensional varieties in $\RR^d$, one expects the main term in the incidence bound to be\footnote{To intuitively see how these exponents are obtained, write $T_{k,d}(m,n)=m^\alpha n^\beta$.
For a constant $r$, we use polynomial partitioning to divide the space into $O(r^d)$ cells, each containing at most $\frac{m}{r^d}$ points and intersecting about $\frac{n}{r^{d-k}}$ varieties on average.
We inductively apply the incidence bound separately in each cell, intuitively leading to the relation $r^d \cdot T_{k,d}(\frac{m}{r^d},\frac{n}{r^{d-k}})\approx T_{k,d}(m,n)$.
For the powers of $r$ to cancel out, we require $d\alpha+(d-k)\beta=d$.
Additionally, in the proof we may assume that $n=O(m^s)$, and we require that $n = O(m^\alpha n^\beta)$.
That is, we intuitively ask that $T_{k,d}(m,m^s) \approx m^s$, or $\alpha+s\beta=s$.
Solving the two equations yields the asserted exponents.}
$$T_{k,d}(m,n):=m^{\frac{sk}{ds-d+k}}n^{\frac{ds-d}{ds-d+k}}.$$
Note that this is indeed the main term in all three parts of Theorem \ref{th:UpperBounds}, up to sub-polynomial factors.
We believe that this bound is tight when $s=2$ (up to sub-polynomial factors).
Moreover, if stronger bounds exist, deriving these is likely to require significantly different techniques.
For the case of curves in $\RR^2$ and $\RR^3$ with $s>2$, one can obtain stronger bounds using the technique of cutting curves into pseudo-segments (see for example \cite{SZ17,Zahl17}).

\parag{Our main result.} Our main result is a general incidence bound for $k$-dimensional varieties in $\RR^d$ with no $K_{s,t}$ in the incidence graph.
Our bound has the main term $m^{\frac{sk}{ds-d+k}+\eps}n^{\frac{ds-d}{ds-d+k}}$ where $\eps$ is an arbitrarily small positive number.

The intuition behind our proofs is based on the concept of \emph{dimension ratio}.
When studying incidences between points and $k$-dimensional varieties in $\RR^d$, we define the dimension ratio of the problem as $\frac{k}{d}$.
As shown in the following lemma, the smaller the dimension ratio is, the smaller the expected leading term is.
For a proof of this lemma, see Section \ref{sec:Rd}.

\begin{lemma}[Dimension ratio lemma]\label{le:ratio test}
Consider positive integers $n,m,s,k,d,k',d'$, such that $n=O(m^s)$, $s>1$, and $\frac{k'}{d'}\le \frac{k}{d} < 1$.
Then $T_{k',d'}(m,n)=O(T_{k,d}(m,n))$.
\end{lemma}

When handling incidences with $k$-dimensional varieties in $\RR^d$ using polynomial partitioning, the analysis usually involves incidence problems in lower dimensions or with lower-dimensional varieties (inside the varieties that partition the space).
When reaching a subproblem with $k'$-dimensional varieties in $\RR^{d'}$ such that $k'/d' \le k/d$, Lemma \ref{le:ratio test} states that the incidence bound of the subproblem is subsumed by the main incidence bound.
That motivates the following definition.
For integers $1\le k < d$, let
\[ R_{k,d} = \left\{(k',d')\in \ZZ^2\ :\ 1\le k'\le k,\ 2\le d' \le d,\ \frac{k}{d}< \frac{k'}{d'} < 1\right\}. \]
That is, $R_{k,d}$ is the set of all ``problematic'' ratios when studying incidences with $k$-dimensional varieties in $\RR^d$.

For example, consider the problem of studying incidences with three-dimensional varieties in $\RR^5$.
This problem has a dimension ratio of $3/5$, and $R_{3,5}=\{(2,3),(3,4)\}$.
When part of the analysis leads to incidences with curves in $\RR^2$, we expect this case to be easy to handle since $3/5>1/2$.
We should be more worried about incidences with two-dimensional varieties in $\RR^3$ and with three-dimensional varieties in $\RR^4$ (since $2/3>3/5$ and $3/4>3/5$).
To handle these cases, we can add the assumption that no bounded-degree four-dimensional variety contains more than $q_{3,4}$ of our three-dimensional varieties, and that no bounded-degree three-dimensional variety has a two-dimensional intersection with more than $q_{2,3}$ of the varieties. By presenting lower bound constructions, we demonstrate that such additional restrictions are often necessary.
When the above example does not contain the restriction involving $q_{2,3}$, the number of incidences could be $T_{2,3}(m,n)$, which is asymptotically larger than $T_{5,3}(m,n)$ (see Theorem \ref{th:LowerMain} below).

Due to the above, one might expect our incidence bound to be of the form
\begin{equation}\label{eq:(k,d)=(3,5)}
O\left(T_{3,5}(m,n)+m^*n^*q_{2,3}^*+m^*n^*q_{3,4}^*+m+n\right),
\end{equation}
for certain choices of exponents $*$.
Unfortunately, the situation is more involved.
For example, consider the case of incidences with four-dimensional objects in $\RR^7$.
Since $\frac{4}{7}<\frac{3}{5}$, we assume that every bounded-degree five-dimensional variety has a three-dimensional intersection with at most $q_{3,5}$ of the four-dimensional varieties.
Then, when bounding the number of incidences inside the partition, we need to bound the number of incidences between $m$ points and $q_{3,5}$ three-dimensional sub-varieties.
Relying on the bound in \eqref{eq:(k,d)=(3,5)} with $n$ replaced by $q_{3,5}$ leads to terms of form $m^*n^*q_{3,5}^*q_{2,3}^*$ and $m^*n^*q_{3,5}^*q_{3,4}^*$.

The above leads us to make the following definition.
A sequence of pairs of positive integers $( (k_0,d_0), (k_1,d_1),\dots, (k_u, d_u) )$ is said to be \emph{significant} if
\begin{enumerate}[(i)]
\item $k_j<d_j$ for all integers $0\le j\le u$.
\item $k_0\geq k_1\geq\dots\geq k_u\geq 1$.
\item  $d_0>d_1>\dots> d_u\geq 2$.
\item $\frac{k_0}{d_0}<\frac{k_1}{d_1}<\dots<\frac{k_u}{d_u} <1$.
\end{enumerate}

Let $S_{k,d}$ be the set of significant sequences $( (k_0,d_0), (k_1,d_1),\dots, (k_u, d_u) )$ with $(k_0,d_0) = (k,d)$.
Note that $((k,d))$ is also a sequence in $S_{k,d}$.
We are finally ready to rigorously state our main result.
For the proof, see Section \ref{sec:Rd}.

\begin{theorem} \label{th:IncReal}
Let $k,d,s,t,D$ be positive integers with $s\geq 2$ and $k<d$.
For any $\eps>0$, there exists a constant $c$ such that the following holds.
Let $\pts$ be a set of $m$ points and $\vars$ be a set of $n$ irreducible varieties of dimension at most $k$ and degree at most $D$, both  in $\RR^d$.
Assume that the incidence graph of $\pts\times \vars$ contains no copy of $K_{s,t}$.
Moreover, for each pair $(k',d')\in R_{k,d}$, every $d'$-dimensional variety of degree at most $c$ has a $k'$-dimensional intersection with at most $q_{k',d'}$ varieties of $\vars$.
Then
\begin{align}
I(\pts,\vars)=O\Bigg( \sum_{( (k,d),(k_1,d_1), \ldots, (k_u, d_u) ) \in S_{k,d}} \hspace{-13mm}
m^{\frac{sk_u}{sd_u-d_u+k_u}+\eps}\Bigg(n^{\frac{d}{d-k}}q_{k_1,d_1}^{\frac{d_1}{d_1-k_1}-\frac{d}{d-k}}\cdots &q_{k_u,d_u}^{\frac{d_u}{d_u-k_u}-\frac{d_{u-1}}{d_{u-1}-k_{u-1}}}\Bigg)^{\frac{(d_u-k_u)(s-1)}{sd_u-d_u+k_u}} \nonumber \\
&\hspace{14mm}+ m + n \Bigg). \label{eq:generalBound}
\end{align}
\end{theorem}

\parag{Remarks.} (i) The sequence $((k,d)) \in S_{k,d}$ results in \eqref{eq:generalBound} containing the leading term $T_{k,d}(m,n)$.  \\[2mm]
(ii) Both $c$ and the constant hidden in the $O(\cdot)$-notation of \eqref{eq:generalBound} depend on $k,d,s,t,D,\eps$.\\[2mm]
(iii) When $k=d-1$ then $R_{k,d}$ is empty and $S_{k,d}=\{((k,d))\}$. In this case no parameter $q_{k',d'}$ is necessary, so Theorem \ref{th:IncReal} generalizes Theorem \ref{th:UpperBounds}(b). \\[2mm]
(iv) When $k=d/2$, Theorem \ref{th:IncReal} has the same leading term as Theorem \ref{th:UpperBounds}(a), but with the additional terms depending on the parameters $q_{k',d'}$.
This is because Theorem \ref{th:IncReal} does not make the transversality assumption that appears in Theorem \ref{th:UpperBounds}(a) and also counts incidences with singular points. See Theorem \ref{th:RdTrans} below for a proper generalization of Theorem \ref{th:UpperBounds}(a). \\[2mm]
(v) When $k=1$, Theorem \ref{th:IncReal} has the same leading term as Theorem \ref{th:UpperBounds}(c), but with a different dependency on the parameters $q_{k',d'}$. Theorem \ref{th:UpperBounds}(c) has a simpler dependency in these parameters, but also has additional restrictions regarding them. The different dependencies are obtained by relying on properties that are special to curves.

\parag{Lower bounds.}
The following theorem shows that the main term in the bound of Theorem \ref{th:UpperBounds} is tight up to sub-polynomial factors, when $s=2$ and for specific values of $k$ and $d$.
It also provides non-trivial bounds for the case of $s=3$.

\begin{theorem} \label{th:LowerMain}
For any integer $d\ge 2$ there exists a sufficiently large constant $t$ satisfying the following claims for every $\eps>0$:

(a) For any $n$ and $m=O(n^d)$, there exist a set $\pts$ of $m$ points and a set $\hyperp$ of $n$ hyperplanes, both in $\RR^d$, such that the incidence graph of $\pts\times \hyperp$ contains no $K_{2,t+1}$ and
\[ I(\pts,\hyperp) = \Omega\left(m^{\frac{2d-2}{2d-1}}n^{\frac{d}{2d-1}-\eps} +m + n\right).\]

(b) For any $d\ge 4$, $n$, and $m=O(n^{d-2})$, there exist a set $\pts$ of $m$ points and a set $\hyperp$ of $n$ hyperplanes, both in $\RR^d$, such that the incidence graph of $\pts\times \hyperp$ contains no $K_{3,t+1}$ and
\[ I(\pts,\hyperp) = \Omega\left(m^{\frac{3d^2-9d+2}{(d-2)(3d-1)}} n^{\frac{2d}{3d-1}-\eps} +m + n\right).\]
\end{theorem}
\parag{Remarks.} (i) Theorem \ref{th:LowerMain} is stated for flats, but it is not difficult to extend it to many other types of varieties. By using the same approach as in \cite{Sheffer16}, one obtains that all three parts of the theorem also hold for spheres, paraboloids, and many other families of varieties of any constant degree. \\[2mm]
(ii) While the bound in part (b) does not match the corresponding upper bound $m^{\frac{3d-3}{3d-1}}n^{\frac{2d}{3d-1}}$, both bounds approach $mn^{2/3}$ as $d$ increases. \\[2mm]
(iii) Several cases of part (a) of Theorem \ref{th:LowerMain} were previously known. When $d=2$, this is the standard lower bound of the Szemer\'edi--Trotter theorem (for example, see \cite{PS04}). The bound was derived in $\RR^3$ by Apfelbaum and Sharir \cite{AS07}. It was also derived for any $d\ge 4$ in \cite{Sheffer16}, but only when $n =\Theta(m^{3/(d+1)})$ up to
sub-polynomial factors. \\[2mm]

After considering lower bounds for the main term in the bound of Theorem \ref{th:LowerMain}, we move to the other terms of that bound.
We show that when $s=2$, this bound must have terms containing various parameters $q_{k',d'}$.
While the dependency in these parameters cannot be removed, it seems likely that it could be replaced with a better dependency.
Such dependency is already known for a few cases involving curves in dimensions $2\le d \le 4$ (for example, see \cite{ShSo17,ShSo17b}).

\begin{theorem} \label{th:LowerOther}
Consider positive integers $k,d,d'$ that satisfy $(d'-1)/d' > k/d$.
Then it is impossible to completely remove the dependency in $q_{d'-1,d'}$ from the incidence bound of Theorem \ref{th:IncReal} in the case of $k$-flats in $\RR^d$ and $s=2$.
\end{theorem}

Proofs for Theorems \ref{th:LowerMain} and \ref{th:LowerOther} can be found in Section \ref{sec:Lower}.

\parag{Incidences with transverse varieties.}
Let $\vars$ be a set of varieties in $\RR^d$.
We say that the varieties of $\vars$ are \emph{transverse} if every two varieties $h_1,h_2\in \vars$ have the following property: If $p$ is a regular point of both $h_1$ and $h_2$, then the intersection of the tangent flats of $h_1$ and $h_2$ at $p$ contains only the origin (we consider a tangent flat as a linear subspace).
Note that this definition is not very interesting for varieties of dimension larger than $d/2$, since such transverse varieties can intersect only in singular points.

Let $h$ and $h'$ be $k$-dimensional transverse varieties in $\RR^d$, and let $U$ be a $d'$-dimensional variety.
If $U$ has a $k'$-dimensional intersection with both $h$ and $h'$ where $k'>d'/2$, then every point of $h\cap h'\cap U$ is a singular point of at least one of $U,h,h',h\cap U,$ and $h'\cap U$.
Intuitively, this implies that we should mainly worry about $(k',d')\in R_{k,d}$ that satisfy $\frac{k'}{d'}\leq\frac{1}{2}$.
Let $\overline{R}_{k,d}$ be the set of pairs $(k',d')\in R_{k,d}$ satisfying $k'/d' \le 1/2$.
Let $\overline{S}_{k,d}$ be the set of sequences of $S_{k,d}$ that consist only of elements of $\overline{R}_{k,d}$.
We derive the following bound for incidences with transverse varieties.

\begin{theorem} \label{th:RdTrans}
Let $k,d,s,t,D$ be positive integers with $s\geq 2$ and $k\le d/2$.
For any $\eps>0$, there exists a constant $c$ such that the following holds.
Let $\pts$ be a set of $m$ points and let $\vars$ be a set of $n$ irreducible varieties of dimension at most $k$ and degree at most $D$, both  in $\RR^d$.
Assume that the incidence graph of $\pts\times \vars$ contains no copy of $K_{s,t}$, and that the varieties of $\vars$ are transverse.
Moreover, for each pair $(k',d')\in \overline{R}_{k,d}$, every $d'$-dimensional variety of degree at most $c$ has a $k'$-dimensional intersection with at most $q_{k',d'}$ varieties of $\vars$.
Then
\begin{align*}
I^*(&\pts,\vars) \\
&=O\Bigg( \sum_{( (k_0,d_0), \ldots, (k_u, d_u) ) \in \overline{S}_{k,d}} \hspace{-13mm}
m^{\frac{sk_u}{sd_u-d_u+k_u}+\eps}\Bigg(n^{\frac{d}{d-k}}q_{k_1,d_1}^{\frac{d_1}{d_1-k_1}-\frac{d}{d-k}}\cdots q_{k_u,d_u}^{\frac{d_u}{d_u-k_u}-\frac{d_{u-1}}{d_{u-1}-k_{u-1}}}\Bigg)^{\frac{(d_u-k_u)(s-1)}{sd_u-d_u+k_u}} \nonumber \\
&\hspace{125mm}+ m + n \Bigg).
\end{align*}
\end{theorem}

We note that Theorem \ref{th:RdTrans} implies Theorem \ref{th:UpperBounds}(a).
Indeed, each term in the bound of Theorem \ref{th:RdTrans} is bounded by $T_{k_u,d_u}(m,n)$, which is in turn bounded by $T_{1,2}(m,n)=m^{\frac{s}{2s-1}}n^{\frac{2s-2}{2s-1}}$ (by Lemma \ref{le:ratio test}).
For a proof of Theorem \ref{th:RdTrans}, see Section \ref{sec:Trans}.
This theorem is tight up to sub-polynomial factors when $s=2$ and $k=d/2$.

\parag{Incidences in $\CC^d$.}
By relying on the dimension ratio approach, we also prove a bound for incidences with hyperplanes in $\CC^d$.
When $s=2$, Theorem \ref{th:LowerMain} implies that this bound is tight up to sub-polynomial factors (the construction in $\RR^d$ can be placed in $\CC^d$ without any changes).
As far as we know, this is the first tight incidence bound in a complex space that is not for lines.

\begin{theorem} \label{th:Cd}
Let $s,t \ge 2$ be integers.
Let $\pts$ be a set of $m$ points and let $\vars$ be a set of $n$ hyperplanes, both in $\CC^d$.
Assume that the incidence graph of $\pts\times \vars$ contains no copy of $K_{s,t}$.
Then for any $\eps>0$ we have
\begin{align*}
I(\pts,\vars)=O\left( m^{\frac{(d-1)s}{ds-1}+\eps}n^{\frac{ds-d}{ds-1}}+ m + n \right).
\end{align*}
\end{theorem}

A proof of Theorem \ref{th:Cd} can be found in Section \ref{sec:Complex}.
\vspace{2mm}

\parag{Acknowledgements.}
We would like to thank Larry Guth for several helpful discussions.

\section{Preliminaries}\label{sec:prelim}

We rely on the following variant of the Zarankiewicz problem (for example, see \cite[Section 4.5]{mat02}).

\begin{lemma}[K\H{o}v\'ari, S\'os and Tur\'an ] \label{le:KSTgen}
Let $\pts$ be a set of $m$ points in $\RR^d$ and let $\vars$ be a set of $n$ subsets of $\RR^d$.
If the incidence graph of $\pts\times\vars$ contains no copy of $K_{s,t}$, then
\[ I(\pts,\vars) = O_{s,t}\left(mn^{1-\frac{1}{s}}+n\right). \]
\end{lemma}

\parag{Varieties and partitioning.}
The variety defined by the polynomials $f_1,\ldots,f_k\in \RR[x_1,\ldots,x_d]$ is
\[ \vb(f_1,\ldots,f_k) = \left\{(a_1,\ldots,a_d)\in \RR^d  : f_j(a_1,\ldots,a_d)=0  \text{ for all }  1\le j \le k\right\}. \]

There are several non-equivalent definitions for the degree of a variety in $\RR^d$.
For our purposes, we define the degree of a variety $U\subset \RR^d$ as
\begin{equation} \label{eq:DegreeRd}
\min_{f_1,\ldots,f_k \in \RR[x_1,\ldots,x_d] \atop \vb(f_1,\ldots,f_k) = U} \max_{1 \le i \le k} \deg f_i.
\end{equation}
That is, the degree of $U$ is the minimum integer $D$ such that $U$ can be defined with a finite set of polynomials of degree at most $D$.

Intuitively, we say that a variety $U\subset \RR^d$ has dimension $k$ if there exists a subset of $U$ that is homeomorphic to the open $k$-dimensional cube, but no subset of $U$ is homeomorphic to an open cube of a larger dimension.
We refer to a $k$-dimensional variety of degree one (or a $k$-dimensional ``plane'') as a $k$-\emph{flat}.
For more information about varieties in $\RR^d$ and a more precise definition of dimension, see for example \cite{BCR98}.

We will use the following variant of the polynomial partitioning theorem.

\begin{theorem}[\cite{FPSSZ17}]\label{th:partition}
Let $\pts$ be a set of $m$ points in $\RR^d$ and let $U\subset\RR^d$ be an irreducible variety of degree $k$ and dimension $d'$. Then for every $1<r<m$ exists $f\in \RR[x_1,\ldots,x_d]$ of degree $O_{d,k}(r)$ such that $U \not\subseteq \vb(f)$ and every connected component of $U\setminus \vb(f)$ contains at most $m/r^{d'}$ points of $\pts$.
\end{theorem}

We will also require the following results about irreducible components and connected components of varieties in $\RR^d$.

\begin{lemma} \label{le:BoundedNumComponents}
Let $U\subset \RR^d$ be a variety of degree $k$.
Then the number of irreducible components of $U$ is $O_{d,k}(1)$.
\end{lemma}

\begin{theorem}[Barone and Basu \cite{BB12}] \label{th:BaroneBasu}
Let $U$ and $W$ be varieties in $\RR^d$ such that $W$ is defined by a single polynomial of degree $k_W \ge 2\deg U$.
Then the number of connected components of $U\setminus W$ is $O_{d}\left(k_W^{\dim U}\deg U^{d-\dim U}\right)$.
\end{theorem}

\parag{Singular points, regular points, and tangent flats.}
The \emph{ideal} of a variety $U\subseteq \RR^d$, denoted $\ib(U)$, is the set of polynomials in $\RR[x_1,\ldots,x_d]$ that vanish on every point of $U$.
We say that a set of polynomials $f_1,\ldots,f_\ell \in \RR[x_1,\ldots,x_d]$ \emph{generate} $\ib(U)$ if every element of $\ib(U)$ can be written as $\sum_{j=1}^\ell f_j g_j$ for some $g_1,\ldots,g_\ell \in \RR[x_1,\ldots,x_d]$.
We also write $\langle f_1,\ldots,f_\ell \rangle = \ib(U)$ to state that $f_1,\ldots,f_\ell$ generate $\ib(U)$.

The \emph{Jacobian matrix} of a set of polynomials $f_1,\ldots,f_k \in \RR[x_1,\ldots,x_d]$ is

\[ \jb_{f_1,\ldots,f_k} = \left( \begin{array}{cccc}
\frac{\partial f_1}{\partial x_1} & \frac{\partial f_1}{\partial x_2} & \cdots & \frac{\partial f_1}{\partial x_d} \\[2mm]
\frac{\partial f_2}{\partial x_1} & \frac{\partial f_2}{\partial x_2} & \cdots & \frac{\partial f_2}{\partial x_d} \\[2mm]
\cdots & \cdots & \cdots & \cdots \\[2mm]
\frac{\partial f_k}{\partial x_1} & \frac{\partial f_k}{\partial x_2} & \cdots & \frac{\partial f_k}{\partial x_d} \end{array}\right)\]

Consider a variety $U\subset\RR^d$ of dimension $k$, and let $f_1,\ldots,f_\ell \in \RR[x_1,\ldots,x_d]$ satisfy $\langle f_1,\ldots,f_\ell \rangle = \ib(U)$.
We say that $p\in U$ is a \emph{singular} points of $U$ if $\rank \jb(p) <d-k$.
We denote the set of singular points of $U$ as $U_\text{sing}$.
A point of $U$ that is not singular is said to be a \emph{regular} point of $U$.
A $k$-dimensional variety has a unique well-defined tangent $k$-flat at every regular point.
We denote the tangent $k$-flat at $p\in U$ as $T_p U$, and think of it as a linear subspace (that is, as incident to the origin).
At singular points of a variety, a unique well-defined tangent flat may or may not exist.

\begin{theorem} \label{th:Singular}
Let $U\subset\RR^d$ be a variety of degree $D$ and dimension $k$.
Then $U_\text{sing}$ is a variety of dimension smaller than $k$ and of degree $O_{D,d}(1)$.
\end{theorem}

References for the above claims and additional information can be found, for example, in \cite{BCR98}.

\section{Upper bounds in $\RR^d$} \label{sec:Rd}

The purpose of this section is to prove Theorem \ref{th:IncReal} --- our general incidence bound in $\RR^d$.
We begin by proving Lemma \ref{le:ratio test}, and first repeat the statement of this lemma.
\vspace{2mm}

\noindent {\bf Lemma \ref{le:ratio test}} (Dimension ratio lemma).
\emph{Consider positive integers $n,m,s,k,d,k',d'$, such that $n=O(m^s)$, $s>1$, and $\frac{k'}{d'}\le \frac{k}{d} < 1$.
Then} $T_{k',d'}(m,n) = O\left(T_{k,d}(m,n)\right)$.

\begin{proof}
By the assumptions, both $s-1$ and $d'k-dk'$ are positive.
We may thus raise both sides of $n=O(m^s)$ to the power of $(d'k-dk')(s-1)$, obtaining
\[ n^{(d'k-dk')(s-1)} = O\left(m^{s(d'k-dk')(s-1)}\right). \]
This implies
\begin{align*}
m^{sk'(ds-d+k)}n^{(d's-d')(ds-d+k)} &= m^{sk'(ds-d+k)}n^{(d'k-dk')(s-1)}n^{(ds-d)(d's-d'+k')}\\[2mm]
&= O\left(m^{sk'(ds-d+k)}m^{s(d'k-dk')(s-1)}n^{(ds-d)(d's-d'+k')}\right)\\[2mm]
&=O\left(m^{sk(d's-d'+k')}n^{(ds-d)(d's-d'+k')}\right).
\end{align*}
Finally, raising both sides to the power of $1/(ds-d+k)(d's-d'+k')$ yields the assertion of the lemma.
\end{proof}

Instead of proving Theorem \ref{th:IncReal}, we prove the following more general result, where the points are contained in a constant-degree variety $W\subset \RR^{d^*}$.
Theorem \ref{th:IncReal} is immediately obtained by setting $W=\RR^d$ and $d^*=d$ in Theorem \ref{th:IncRealGen}.

\begin{theorem} \label{th:IncRealGen}
Let $k,d,d^*, s,t,D_W$ be positive integers with $s\geq 2$ and $k<d$.
For any $\eps>0$, there exists a constant $c$ such that the following holds.
Let $\pts$ be a set of $m$ points on an irreducible $d$-dimensional variety $W\subseteq\RR^{d^*}$ of degree $D_W$ and let $\vars$ be a set of $n$ irreducible varieties of dimension at most $k$ and degree at most $D$ in $\RR^{d^*}$.
Assume that the incidence graph of $\pts\times \vars$ contains no copy of $K_{s,t}$.
Moreover, for each pair $(k',d')\in R_{k,d}$, every $d'$-dimensional variety of degree at most $c$ has a $k'$-dimensional intersection with at most $q_{k',d'}$ varieties of $\vars$.
Then
\begin{align*}
I(\pts,\vars)=O\Bigg( \sum_{( (k_0,d_0), \ldots, (k_u, d_u) ) \in S_{k,d}} \hspace{-13mm}
m^{\frac{sk_u}{sd_u-d_u+k_u}+\eps}\Bigg(n^{\frac{d}{d-k}}q_{k_1,d_1}^{\frac{d_1}{d_1-k_1}-\frac{d}{d-k}}\cdots &q_{k_u,d_u}^{\frac{d_u}{d_u-k_u}-\frac{d_{u-1}}{d_{u-1}-k_{u-1}}}\Bigg)^{\frac{(d_u-k_u)(s-1)}{sd_u-d_u+k_u}} \\
&\hspace{21mm}+ m + n \Bigg).
\end{align*}
\end{theorem}
\begin{proof}
We prove the statement of the theorem by induction on $d$.
In particular, we prove that
\begin{align}
I(\pts,\vars) \le \alpha_{1} \cdot \hspace{-13mm}\sum_{( (k_0,d_0), \ldots, (k_u, d_u) ) \in S_{k,d}} \hspace{-13mm}
m^{\frac{sk_u}{sd_u-d_u+k_u}+\eps}\Bigg(n^{\frac{d}{d-k}}q_{k_1,d_1}^{\frac{d_1}{d_1-k_1}-\frac{d}{d-k}}\cdots &q_{k_u,d_u}^{\frac{d_u}{d_u-k_u}-\frac{d_{u-1}}{d_{u-1}-k_{u-1}}}\Bigg)^{\frac{(d_u-k_u)(s-1)}{sd_u-d_u+k_u}} \nonumber\\ &\hspace{26mm}+\alpha_{2}(m+n), \label{eq:FirstInd}
\end{align}
where $\alpha_{1},\alpha_{2}$ are sufficiently large constants that depend on $s,t,c,d,d^*,D_W,$ and $\eps$.

For the induction basis, we consider the case where $d=2$.
In this case $R_{1,2}$ is empty and $S_{1,2}=\{(k,d)\}$.
Let $h$ be a generic two-dimensional plane in $\RR^{d^*}$, and let $\pi_2:\RR^{d^*}\to h$ be the standard projection.
Set $\pts_2 = \{\pi_2(p) \ :\ p\in \pts\}$ and let $\curves$ be the set of projections of the curves of $\vars$.
Since $h$ is chosen generically, we may assume that no two points of $\pts$ are projected to the same point of $h$, that no new incidences were created due to the projection, and that every element of $\curves$ is a constant-degree variety of dimension one.
This implies that the incidence graph of $\pts_2\times \curves$ contains no copy of $K_{s,t}$ and that $I(\pts,\vars)=I(\pts_2,\curves)$.
The theorem is then obtained by applying Theorem \ref{th:UpperBounds}(b) to $\pts_2$ and $\curves$.

We prove the induction step using a second induction on $m+n$.
For the base case of this second induction, when $m$ and $n$ are sufficiently small the result is obtained by choosing sufficiently large values of $\alpha_{1}$ and $\alpha_{2}$.
It remains to handle the induction step of the second induction.
The hidden constants in the $O(\cdot)$-notations throughout the proof may also depend on $s,t,D,d,d^*,D_W$ and $\eps$.
For brevity we write $O(\cdot)$ instead of $O_{s,t,D,d,d^*,\deg W,\eps}(\cdot)$.

Since the incidence graph contains no copy of $K_{s,t}$, Lemma \ref{le:KSTgen} implies $I(\pts,\vars) = O(mn^{1-1/s}+n)$.
When $m=O(n^{1/s})$, this implies $I(\pts,\vars) = O(n)$.
We may thus assume that
\begin{equation} \label{eq:kstAssum}
n=O(m^s).
\end{equation}

We now present a brief outline of the rest of the proof.
We use a bounded-degree partitioning polynomial $f$ to partition $W$ into cells, each containing a bounded number of points of $\pts$.
We apply the second induction hypothesis separately in each cell, and sum these bounds in a standard way.
To handle incidences on the partition $Z(f)\cap W$,  we separately consider each irreducible component of this intersection.
Let $W'$ be such an irreducible component of dimension $d'$.
For every $1\le k' \le k$, we separately consider the varieties of $\vars$ that have a $k'$-dimensional intersection with $W'$.
By applying the first induction hypothesis, we obtain a bound for the number of incidences with a term corresponding to every significant sequence of $S_{k',d'}$.
We extend each such sequence to a significant sequence  of $S_{k,d}$.
For example, when $\frac{k}{d}<\frac{k'}{d'}$, we simply add $(k,d)$ to the beginning of the sequence.
By carefully handling the various types of sequences of $S_{k',d'}$, we always obtain terms that are part of the bound of the induction, which completes the induction step.

\parag{Partitioning the space.}
By Theorem \ref{th:partition} with $U=W$, $d'=d$, and a constant $r$, we obtain a polynomial $f\in \RR[x_1,\ldots,x_{d^*}]$ of degree $O(r)$ such that every connected component of $W\setminus \vb(f)$ contains at most $m/r^d$ points of $\pts$.
The asymptotic relations between the various constants in the proof are
\[ 2^{1/\eps} \ll r \ll \alpha_{2} \ll \alpha_{1} \quad \text{ and } \quad k,d,d^*, s,t,D,D_W,2^{1/\eps} \ll c.\]

Denote the cells of the partition as $C_1, \ldots, C_v$.
By Theorem \ref{th:BaroneBasu}, we have that $v=O(r^{d})$.
For each $1\le j \le v$, denote by $\vars_j$ the set of varieties of $\vars$ that intersect $C_j$, and set $\pts_j = C_j \cap \pts$.
We also set $m_j=|\pts_j|$, $m' = \sum_{j=1}^c m_j$, and $n_j=|\vars_j|$.
Note that $m_j\le m/r^{d}$ for every $1\le j \le v$.
By Theorem \ref{th:BaroneBasu}, every variety of $\vars$ intersects $O(r^k)$ cells of $W \setminus \vb(f)$.
Therefore, $\sum_{j=1}^v n_j = O(nr^k)$.
For every $(k_u,d_u)\in R_{k,d}$, H\"older's inequality implies
\begin{align*}
\sum_{j=1}^v n_j^{\frac{d(d_u-k_u)(s-1)}{(d-k)(sd_u-d_u+k_u)}}&\le \left(\sum_{j=1}^v n_j\right)^{\frac{d(d_u-k_u)(s-1)}{(d-k)(sd_u-d_u+k_u)}} \left(\sum_{j=1}^v 1\right)^{\frac{dk_us-d_uks+kd_u-kk_u}{(d-k)(sd_u-d_u+k_u)}} \\[2mm]
&= O\left(\left(nr^k\right)^{\frac{d(d_u-k_u)(s-1)}{(d-k)(sd_u-d_u+k_u)}} r^{\frac{d(dk_us-d_uks+kd_u-kk_u)}{(d-k)(sd_u-d_u+k_u)}}\right) \\[2mm]
&=O\left(n^{\frac{d(d_u-k_u)(s-1)}{(d-k)(sd_u-d_u+k_u)}} r^{\frac{dk_us}{sd_u-d_u+k_u}}\right).
\end{align*}

Combining the above with the induction hypothesis implies
\begin{align*}
\sum_{j=1}^v &I(\pts_j,\vars_j) \\
&\le \sum_{j=1}^v \Bigg(\alpha_{1} \cdot \hspace{-6mm}\sum_{( (k_0,d_0), \ldots, (k_u, d_u) ) \in S_{k,d}} \hspace{-13mm}
m_j^{\frac{sk_u}{sd_u-d_u+k_u}+\eps}\Bigg(n_j^{\frac{d}{d-k}}q_{k_1,d_1}^{\frac{d_1}{d_1-k_1}-\frac{d}{d-k}}\cdots q_{k_u,d_u}^{\frac{d_u}{d_u-k_u}-\frac{d_{u-1}}{d_{u-1}-k_{u-1}}}\Bigg)^{\frac{(d_u-k_u)(s-1)}{sd_u-d_u+k_u}}\\
&\hspace{108mm}+\alpha_{2}(m_j+n_j)\Bigg) \\
&\le  \alpha_{1} \cdot \hspace{-14mm}\sum_{( (k_0,d_0), \ldots, (k_u, d_u) ) \in S_{k,d}}  \frac{m^{\frac{sk_u}{sd_u-d_u+k_u}+\eps}}{r^{\frac{sk_ud}{sd_u-d_u+k_u}+d\eps}}\Bigg(q_{k_1,d_1}^{\frac{d_1}{d_1-k_1}-\frac{d}{d-k}}\cdots q_{k_u,d_u}^{\frac{d_u}{d_u-k_u}-\frac{d_{u-1}}{d_{u-1}-k_{u-1}}}\Bigg)^{\frac{(d_u-k_u)(s-1)}{sd_u-d_u+k_u}} \\
&\hspace{72mm}\cdot \sum_{j=1}^v n_j^{\frac{d(d_u-k_u)(s-1)}{(d-k)(sd_u-d_u+k_u)}} +\alpha_{2}\left(m'+O\left(nr^k\right)\right) \\
&\le  \alpha_{1} \cdot \hspace{-14mm}\sum_{( (k_0,d_0), \ldots, (k_u, d_u) ) \in S_{k,d}}  \frac{m^{\frac{sk_u}{sd_u-d_u+k_u}+\eps}}{r^{d\eps}}\Bigg(n^{\frac{d}{d-k}}q_{k_1,d_1}^{\frac{d_1}{d_1-k_1}-\frac{d}{d-k}}\cdots q_{k_u,d_u}^{\frac{d_u}{d_u-k_u}-\frac{d_{u-1}}{d_{u-1}-k_{u-1}}}\Bigg)^{\frac{(d_u-k_u)(s-1)}{sd_u-d_u+k_u}}\\
&\hspace{106mm}+\alpha_{2}\left(m'+O\left(nr^k\right)\right).
\end{align*}

By \eqref{eq:kstAssum} we have $n=O\left(m^{\frac{sk}{sd-d+k}}n^{\frac{sd-d}{sd-d+k}}\right)$.
Thus, when $\alpha_{1}$ is sufficiently large with respect to $r$ and $\alpha_{2}$, we get
\begin{align*}
&\sum_{j=1}^v I(\pts_j,\vars_j)  \\
&\hspace{2mm}=O\Bigg(\alpha_{1} \cdot \hspace{-6mm} \sum_{( (k_0,d_0), \ldots, (k_u, d_u) ) \in S_{k,d}}  \hspace{-8mm} \frac{m^{\frac{sk_u}{sd_u-d_u+k_u}+\eps}}{r^{d\eps}}\Bigg(n^{\frac{d}{d-k}}q_{k_1,d_1}^{\frac{d_1}{d_1-k_1}-\frac{d}{d-k}}\cdots q_{k_u,d_u}^{\frac{d_u}{d_u-k_u}-\frac{d_{u-1}}{d_{u-1}-k_{u-1}}}\Bigg)^{\frac{(d_u-k_u)(s-1)}{sd_u-d_u+k_u}}\Bigg)\\
&\hspace{137mm}+\alpha_{2}m'.
\end{align*}

When $r$ is sufficiently large with respect to $\eps$ and to the constant hidden in the $O(\cdot)$-notation, we have
\begin{align}
&\sum_{j=1}^v I(\pts_j,\vars_j) \nonumber \\
&\hspace{2mm}\le \frac{\alpha_{1}}{4} \Bigg(\sum_{( (k_0,d_0), \ldots, (k_u, d_u) ) \in S_{k,d}}  \hspace{-8mm} m^{\frac{sk_u}{sd_u-d_u+k_u}+\eps}\Bigg(n^{\frac{d}{d-k}}q_{k_1,d_1}^{\frac{d_1}{d_1-k_1}-\frac{d}{d-k}}\cdots q_{k_u,d_u}^{\frac{d_u}{d_u-k_u}-\frac{d_{u-1}}{d_{u-1}-k_{u-1}}}\Bigg)^{\frac{(d_u-k_u)(s-1)}{sd_u-d_u+k_u}} \nonumber \\
&\hspace{126mm}+\alpha_{2}m'. \label{eq:incCells}
\end{align}

\parag{Incidences on the partition} It remains to bound the number of incidences with points that lie on $\vb(f)$.
Set $\pts_0= \pts \cap \vb(f)$ and $m_0=|\pts_0|=m-m'$.
We associate each point $p\in \pts_0$ with an arbitrary component of $\vb(f)\cap W$ that $p$ is incident to.
To bound the number of incidences with $\pts_0$, it suffices to separately consider each irreducible component of $\vb(f)\cap W$ and the points that are associated with it.

Let $U$ be an irreducible component of $\vb(f) \cap W$.
Let $\overline{\pts}$ be the set of points of $\pts$ associated with $U$, and set $\overline{m} = |\overline{\pts}|$.
By Theorem \ref{th:partition} we have that $W\not\subseteq\vb(f)$, so $\dim U <d$.
We set $d' = \dim U$.
By taking $c$ to be sufficiently large, we get that $\deg U = \max\{O(r),D_W\} \le c$.

Since the incidence graph contains no $K_{s,t}$, either $\overline{m}<s$ or at most $t$ varieties of $\vars$ contain $U$.
In the latter case, the varieties of $\vars$ that contain $U$ form $O(\overline{m})$ incidences with $\overline{\pts}$.
In the former case, the number of such incidences is $O(n)$.
By Lemma \ref{le:BoundedNumComponents}, the number of incidences between $\overline{\pts}$ and varieties of $\vars$ that have a finite intersection with $U$ is also $O(n)$.
It remains to consider incidences with varieties $h\in \vars$ that satisfy $1\leq \dim(h\cap U)\leq \min\{k,d'-1\}$.
For $1\le k'\le \min\{k,d'-1\}$, let $\vars_{k'}$ denote the set of $k'$-dimensional intersections between $U$ and an element of $\vars$.
Set $n_{k'}=|\vars_{k'}|$ and note that $\sum_{k'} n_{k'}\leq n$.
By definition, for any $(k',d')\in R_{k,d}$ we have $n_{k'}\leq q_{k',d'}$.

For a fixed $1\le k'\le \min\{d'-1,k\}$, we would like to apply the first induction hypothesis with the set of varieties $\vars_{k'}$, $W=U, d=d'$, and $k=k'$.
For every $(k^*,d^*)\in R_{k',d'}$, denote by $\overline{q}_{k^*,d^*}$ the maximum number of varieties of $\vars_{k'}$ that have a $k^*$-dimensional intersection with any $d^*$-dimensional variety of degree at most $c$.
(When choosing $c$ in the statement of the theorem, in addition to taking $c\ge \max\{\deg f, D_W\}$, we set $c$ to be sufficiently large for reapplying the theorem with smaller values of $d$.)
If $(k^*,d^*)\in R_{k,d}$ then $\overline{q}_{k^*, d^*}\leq q_{k^*,d^*}$.
When $(k^*,d^*)\notin R_{k,d}$, we will use the trivial bound $\overline{q}_{k^*,d^*}\leq n_{k'}\leq n$.

We will now use the dependency of $\alpha_1$ and $\alpha_2$ in $d$, and to stress this we change the notation to $\alpha_{1,d}$ and $\alpha_{2,d}$.
By the first induction hypothesis we have
\begin{align}
I(&\overline{\pts},\vars_{k'}) \nonumber \\
&\le  \alpha_{1,d'}\cdot\hspace{-12mm} \sum_{( (k_0,d_0), \ldots, (k_u, d_u) ) \in S_{k',d'}}  \hspace{-8mm} m^{\frac{sk_u}{sd_u-d_u+k_u}+\eps}\Bigg(n_{k'}^{\frac{d'}{d'-k'}}\overline{q}_{k_1,d_1}^{\frac{d_1}{d_1-k_1}-\frac{d'}{d'-k'}}\cdots \overline{q}_{k_u,d_u}^{\frac{d_u}{d_u-k_u}-\frac{d_{u-1}}{d_{u-1}-k_{u-1}}}\Bigg)^{\frac{(d_u-k_u)(s-1)}{sd_u-d_u+k_u}} \nonumber \\
&\hspace{83mm}+\alpha_{2,d'}(\overline{m}+n) + O_r(\overline{m}+n). \label{eq:IncOnPart}
\end{align}

Consider a term from \eqref{eq:IncOnPart} of the form
\begin{equation}\label{eq:GenTerm}
m^{\frac{sk_u}{sd_u-d_u+k_u}+\eps}\Bigg(n_{k'}^{\frac{d'}{d'-k'}}\overline{q}_{k_1,d_1}^{\frac{d_1}{d_1-k_1}-\frac{d'}{d'-k'}}\cdots \overline{q}_{k_u,d_u}^{\frac{d_u}{d_u-k_u}-\frac{d_{u-1}}{d_{u-1}-k_{u-1}}}\Bigg)^{\frac{(d_u-k_u)(s-1)}{sd_u-d_u+k_u}}
\end{equation}

If $\frac{k}{d} \ge \frac{k_u}{d_u}$, we recall that $\overline{q}_{k_j,d_j}\le n_{k'} \le n$ for every $1\le j \le u$.
Combining this with Lemma \ref{le:ratio test} implies that \eqref{eq:GenTerm} is upper bounded by
\[ m^{\frac{sk_u}{sd_u-d_u+k_u}+\eps}n^{\frac{d_us-d_u}{sd_u-d_u+k_u}} = T_{d_u,k_u}(m,n) \cdot m^\eps = O(T_{d,k}(m,n) \cdot m^\eps) = O\left(m^{\frac{sk}{sd-d+k}+\eps}n^{\frac{ds-d}{sd-d+k}}\right). \]

Next, we consider the case when $\frac{k}{d} < \frac{k_u}{d_u}$.
By definition, $(k_0,d_0)=(k',d')$ and $\overline{q}_{k_0,d_0}=n_{k'}$.
Let $j$ be the smallest non-negative integer that satisfies $\frac{k}{d} < \frac{k_j}{d_j}$.
That is,
\[ \frac{k_0}{d_0}<\dots< \frac{k_{j-1}}{d_{j-1}}\leq\frac{k}{d}<\frac{k_j}{d_j}<\dots <\frac{k_u}{d_u}. \]

Note that $((k,d),(k_j,d_j),\ldots,(k_u,d_u))$ is a significant sequence.
We will use the term in \eqref{eq:generalBound} corresponding to this significant sequence to upper bound \eqref{eq:GenTerm}. For every $j'\geq j$ we have $\overline{q}_{k_{j'},d_{j'}}\le q_{k_j,d_j}$.
By also recalling that $\overline{q}_{k_{j'},d_{j'}}\leq n$, we obtain
\begin{align*}
\overline{q}_{k_0,d_0}^{\frac{d_0}{d_0-k_0}} &\overline{q}_{k_1,d_1}^{\frac{d_1}{d_1-k_1}-\frac{d_0}{d_0-k_0}}\cdots  \overline{q}_{k_j,d_j}^{\frac{d_{j}}{d_{j}-k_{j}}-\frac{d_{j-1}}{d_{j-1}-k_{j-1}}} \\
&\le n^{\frac{d_0}{d_0-k_0}} n^{\frac{d_1}{d_1-k_1}-\frac{d_0}{d_0-k_0}}\cdots n^{\frac{d_{j-1}}{d_{j-1}-k_{j-1}}-\frac{d_{j-2}}{d_{j-2}-k_{j-2}}} q_{k_j,d_j}^{\frac{d_{j}}{d_{j}-k_{j}}-\frac{d_{j-1}}{d_{j-1}-k_{j-1}}}\\
&= n^{\frac{d_{j-1}}{d_{j-1}-k_{j-1}}} q_{k_j,d_j}^{\frac{d_{j}}{d_{j}-k_{j}}-\frac{d_{j-1}}{d_{j-1}-k_{j-1}}} \le  n^{\frac{d}{d-k}} q_{k_j,d_j}^{\frac{d_{j}}{d_{j}-k_{j}}-\frac{d}{d-k}}.
\end{align*}

Thus, in this case \eqref{eq:GenTerm} is upper bounded by
\[ m^{\frac{sk_u}{sd_u-d_u+k_u}+\eps}\Bigg(n^{\frac{d}{d-k}}q_{k_j,d_j}^{\frac{d_j}{d_j-k_j}-\frac{d}{d-k}}\cdots q_{k_u,d_u}^{\frac{d_u}{d_u-k_u}-\frac{d_{u-1}}{d_{u-1}-k_{u-1}}}\Bigg)^{\frac{(d_u-k_u)(s-1)}{sd_u-d_u+k_u}} \]

By combining \eqref{eq:IncOnPart} with the two above bounds for \eqref{eq:GenTerm}, we obtain
\begin{align*}
I(&\overline{\pts},\vars_{k'})  \\
&= O_r\Bigg( \alpha_{1,d'}\cdot \hspace{-10mm}\sum_{( (k_0,d_0), \ldots, (k_u, d_u) ) \in S_{k,d}} \hspace{-13mm}
m^{\frac{sk_u}{sd_u-d_u+k_u}+\eps}\Bigg(n^{\frac{d}{d-k}}q_{k_1,d_1}^{\frac{d_1}{d_1-k_1}-\frac{d}{d-k}}\cdots q_{k_u,d_u}^{\frac{d_u}{d_u-k_u}-\frac{d_{u-1}}{d_{u-1}-k_{u-1}}}\Bigg)^{\frac{(d_u-k_u)(s-1)}{sd_u-d_u+k_u}}  \\
&\hspace{112mm}+\alpha_{2,d'}\left(\overline{m}+n\right)\Bigg).
\end{align*}

Lemma \ref{le:BoundedNumComponents} implies that $\vb(f)\cap W$ has $O_r(1)$ irreducible components.
By summing the above bound over each of these components and over every $k'$, and assuming $\alpha_{1,j} \le \alpha_{1,j+1}$ and $\alpha_{2,j} \le \alpha_{2,j+1}$ for every $j$, we obtain
\begin{align*}
I(&\pts_0,\vars)  \\
&= O_r\Bigg( \alpha_{1,d-1}\cdot \hspace{-10mm}\sum_{( (k_0,d_0), \ldots, (k_u, d_u) ) \in S_{k,d}} \hspace{-13mm}
m^{\frac{sk_u}{sd_u-d_u+k_u}+\eps}\Bigg(n^{\frac{d}{d-k}}q_{k_1,d_1}^{\frac{d_1}{d_1-k_1}-\frac{d}{d-k}}\cdots q_{k_u,d_u}^{\frac{d_u}{d_u-k_u}-\frac{d_{u-1}}{d_{u-1}-k_{u-1}}}\Bigg)^{\frac{(d_u-k_u)(s-1)}{sd_u-d_u+k_u}}  \\
&\hspace{112mm}+\alpha_{2,d-1}\left(m_0+n\right)\Bigg).
\end{align*}

By recalling that $n=O\left(m^{\frac{sk}{sd-d+k}}n^{\frac{sd-d}{sd-d+k}}\right)$ and taking $\alpha_{1,d}$ and $\alpha_{2,d}$ to be sufficiently large with respect to $\alpha_{1,d-1},\alpha_{2,d-1},$ and $r$, we obtain
\begin{align*}
I(&\pts_{0},\vars) \\
&\le \frac{\alpha_{1,d}}{2}\cdot \hspace{-2mm}\sum_{( (k_0,d_0), \ldots, (k_u, d_u) ) \in S_{k,d}} \hspace{-13mm}
m^{\frac{sk_u}{sd_u-d_u+k_u}+\eps}\Bigg(n^{\frac{d}{d-k}}q_{k_1,d_1}^{\frac{d_1}{d_1-k_1}-\frac{d}{d-k}}\cdots q_{k_u,d_u}^{\frac{d_u}{d_u-k_u}-\frac{d_{u-1}}{d_{u-1}-k_{u-1}}}\Bigg)^{\frac{(d_u-k_u)(s-1)}{sd_u-d_u+k_u}} \\
&\hspace{125mm}+\alpha_{2}m_0.
\end{align*}

Combining this bound with \eqref{eq:incCells} completes the two induction steps and the proof of the theorem.
\end{proof}

We can now discuss how the exponents in the bound of Theorem \ref{th:IncRealGen} were obtained.
For a significant sequence $((k,d),(k_1,d_1)\dots, (k_u, d_u))\in S_{k,d}$, the exponents in the corresponding term  $m^{\alpha+\varepsilon}n^{\beta}q_{k_1,d_1}^{\beta_1}\cdots q_{k_u, d_u}^{\beta_u}$ should satisfy the system
\begin{equation*}
\begin{cases}
\alpha+ s(\beta+\beta_1+\dots+\beta_u)=s\\
d\alpha+ (d-k)\beta=d\\
d_1 \alpha + (d_1-k_1)(\beta+\beta_1)=d_1 \\
\cdots\\
d_u\alpha+(d_u-k_u)(\beta+\beta_1+\dots+\beta_u)=d_u.
\end{cases}
\end{equation*}

This system leads to the exponents stated in the theorem, so it only remains to explain how these equations were obtained.
After partitioning the space in the proof of Theorem \ref{th:IncRealGen}, we sum up the incidences in the cells.
In this sum, the powers of $r$ in the numerator and denominator cancel out, leaving only $r^{-d\eps}$.
The denominator contains the factor $r^{d(\alpha+\eps)}$ and the numerator contains $r^{k\beta+d(1-\beta)}$ (coming from H\"older's inequality).
Asking these two powers to be equivalent up to the $d\eps$ immediately leads to the second equation of the above system.

Next, we assume that $q_{k_1,d_1}=n$.
In this case, we have $m^{\alpha+\varepsilon}n^{\beta}q_{k_1,d_1}^{\beta_1}\cdots q_{k_u, d_u}^{\beta_u} = m^{\alpha+\varepsilon}n^{\beta+\beta_1}q_{k_2,d_2}^{\beta_2}\cdots q_{k_u, d_u}^{\beta_u}$.
That is, we are in a problem corresponding to the significant sequence $((k,d),(k_2,d_2),\ldots,(k_u,d_u))$, and with $\beta$ replaced with $\beta+\beta_1$.
By repeating the argument from the previous paragraph with the new sequence, we obtain the third equation in the above system.
The next equation in the system is then obtained by setting $q_{k_1,d_1}=q_{k_2,d_2}=n$, and so on.
This leads to all of the equations of the system, except for the first and last ones.

In the proof of Theorem \ref{th:IncRealGen}, just below \eqref{eq:GenTerm}, we consider the case where $q_{k_i,d_i}=n$ for every $1\le i \le u$.
This implies $m^{\alpha+\varepsilon}n^{\beta}q_{k_1,d_1}^{\beta_1}\cdots q_{k_u, d_u}^{\beta_u} = m^{\alpha+\eps}n^{\beta+\beta_1+\cdots+\beta_u}$.
We then ask that $m^{\alpha+\eps}n^{\beta+\beta_1+\cdots+\beta_u} = O(T_{d_u,k_u}(m,n))$.
The first footnote in the introduction states two equations that are required for this bound to hold.
In these two equations, respectively replacing $d,k,\beta$ with $d_u,k_u,\beta+\beta_1+\cdots+\beta_u$ leads to the first and last equations in the above system.

\section{Lower bounds in $\RR^d$.} \label{sec:Lower}

In this section we prove the lower bounds stated in the introduction.
As usual, before each proof we repeat the statement of the theorem.
\vspace{2mm}

\noindent {\bf Theorem \ref{th:LowerMain}.} \emph{For any integer $d\ge 2$ there exists a sufficiently large constant $t$ satisfying the following claims for every $\eps>0$:}

\emph{(a) For any $n$ and $m=O(n^d)$, there exist a set $\pts$ of $m$ points and a set $\hyperp$ of $n$ hyperplanes, both in $\RR^d$, such that the incidence graph of $\pts\times \hyperp$ contains no $K_{2,t+1}$ and}
\[ I(\pts,\hyperp) = \Omega\left(m^{\frac{2d-2}{2d-1}}n^{\frac{d}{2d-1}-\eps} +m + n\right).\]

\emph{(b) For any $d\ge 4$, $n$, and $m=O(n^{d-2})$, there exist a set $\pts$ of $m$ points and a set $\hyperp$ of $n$ hyperplanes, both in $\RR^d$, such that the incidence graph of $\pts\times \hyperp$ contains no $K_{3,t+1}$ and}
\[ I(\pts,\hyperp) = \Omega\left(m^{\frac{3d^2-9d+2}{(d-2)(3d-1)}} n^{\frac{2d}{3d-1}-\eps} +m + n\right).\]

\begin{proof}
(a) To obtain $m$ incidences, we can place all the points of $\pts$ on a single hyperplane of $\hyperp$.
To obtain $n$ incidences, we can set all of the hyperplanes of $\hyperp$ to be incident to a single point of $\pts$.
Thus, we only need to construct a configuration with $\Omega\left(m^{\frac{2d-2}{2d-1}}n^{\frac{d}{2d-1}-\eps}\right)$ incidences.
For the case of $d=2$, see for example \cite{PS04}.
We may thus also assume that $d\ge 3$.

We consider the point set
\[ \pts = \left\{(x_1,\ldots,x_d) \in \ZZ^d\ :\ 0\ \le x_j \le m^{1/d}-1 \right\}.\]

For a parameter $N$ that will be set below, let $\lattice$ be an $N\times\cdots\times N$ section of $\ZZ^d$ centered at the origin.
We say that a nonzero element $v\in \lattice$ is a \emph{primitive vector} if there is no integer $j>1$ and $u\in \lattice$ such that $v = j\cdot u$.
By \cite[Corollary 1]{BSC17}, there exists a subset $V\subset \lattice$ of $\Theta\left(N^{d/(d-1)-\eps'}\right)$ primitive vectors such that any hyperplane in $\RR^d$ contains at most $t$ of these vectors (for a sufficiently large constant $t$).
Let $\hyperp$ be the set of hyperplanes in $\RR^d$ that contain at least one point of $\pts$ and whose normal direction is in $V$ (the sizes of the two vectors may differ).
The dot product of a vector from $V$ and a point of $\pts$ is an integer of size $O\left(Nm^{1/d}\right)$.
Thus, each vector of $V$ corresponds to $O\left(Nm^{1/d}\right)$ hyperplanes of $\hyperp$.
This implies that $|\hyperp|= O\left(N^{(2d-1)/(d-1)-\eps'} m^{1/d} \right)$.

By possibly adding hyperplanes to $\hyperp$, we assure that $|\hyperp|= \Theta\left(N^{(2d-1)/(d-1)-\eps'} m^{1/d} \right)$.
To have $n\approx |\hyperp|$, we set $N=n^{(d-1)/(2d-1-(d-1)\eps')}/m^{(d-1)/d(2d-1-(d-1)\eps')}$.
The assumption $m=O(n^d)$ implies $N=\Omega(1)$.
Every point of $\pts$ is incident to exactly one hyperplane of $\hyperp$ with each normal of $V$.
By taking $\eps'$ to be sufficiently small with respect to $\eps$, we obtain
\[ I(\pts,\hyperp) = mN^{\frac{d}{d-1}-\eps'} = \Theta\left(m^{\frac{2d-2}{2d-1}}\cdot \left(N^{\frac{d}{d-1}-\eps'} m^{\frac{1}{2d-1}}  \right)\right) = \Omega\left(m^{\frac{2d-2}{2d-1}} n^{\frac{d}{2d-1}-\eps}\right). \]

Let $\ell\subset\RR^d$ be a line.
For a hyperplane $h\in \hyperp$ to contain $\ell$, the normal of $h$ must be orthogonal to the direction of $\ell$.
That is, the normal of the hyperplane is in a given linear $(d-1)$-dimensional subspace.
By the choice of $V$, we obtain that at most $t$ hyperplanes of $\hyperp$ contain any given line.
This implies that the incidence graph of $\pts\times\hyperp$ contains no copy of $K_{2,t+1}$.

(b) As in part (a) of the proof, it is simple to obtain $\Theta(m+n)$ incidences, so we only need to construct a configuration with $\Theta\Big(m^{\frac{3d^2-9d+2}{(d-2)(3d-1)}} n^{\frac{2d}{3d-1}-\eps}\Big)$ incidences.

We consider the point set
\[ \pts' = \left\{(x_1,\ldots,x_d) \in \ZZ^d\ :\ 0\ \le x_j \le m^{1/d}-1 \right\}.\]

The distance between a point $(x_1,\ldots,x_d)\in \pts'$ and the origin is $\sqrt{x_1^2+\cdots+x_d^2}$.
Every such distance is the square root of an integer between zero and $d\cdot m^{2/(d-2)}$.
That is, the points of $\pts'$ determine $O\left(d\cdot m^{2/(d-2)}\right)$ distinct distances from the origin.
By the pigeonhole principle, there exists a distance $\delta$ such that $\Omega\left(m^{d/(d-2)}/\left(d\cdot m^{2/(d-2)}\right)\right) = \Omega(m)$  points are at distance $\delta$ from the origin.
In other words, the hypersphere $S_\delta$ centered at the origin and of radius $\delta$ contains $\Omega(m)$ points of $\pts'$. Let $\pts$ be a set of exactly $m$ of these points (if necessary, we can add extra generic points on $S_\delta$).
To recap, $\pts$ is a set of $m$  points with integer coordinates on the hypersphere $S_\delta$.

For a parameter $N$ that will be set below, let $\lattice$ be an $N\times\cdots\times N$ section of $\ZZ^d$ centered at the origin.
By \cite[Corollary 1]{BSC17}, there exists a subset $V\subset \lattice$ of $\Theta\left(N^{2d/(d-1)-\eps'}\right)$ primitive vectors such that any $(d-2)$-flat in $\RR^d$ contains at most $t$ of these vectors (for a sufficiently large constant $t$).
Let $\hyperp$ be the set of hyperplanes in $\RR^d$ that contain at least one point of $\pts$ and whose normal direction is in $V$ (the sizes of the two vectors may differ).

The dot product of a vector from $V$ and a point of $\pts$ is an integer of size $O\left(Nm^{1/(d-2)}\right)$.
Thus, the number of hyperplanes in $\hyperp$ is $O\left(N^{(3d-1)/(d-1)-\eps'} m^{1/(d-2)}\right)$.
By possibly adding generic hyperplanes to $\hyperp$, we can assure that $|\hyperp|= \Theta\left(N^{(2d-1)/(d-1)-\eps'} m^{1/d} \right)$.
To have $n\approx |\hyperp|$, we set $N=n^{(d-1)/(3d-1-(d-1)\eps')}/m^{(d-1)/(d-2)(3d-1-(d-1)\eps')}$.
The assumption $m=O\left(n^{d-2}\right)$ implies that $N=\Omega(1)$.

Every point of $\pts$ is incident to exactly one hyperplane of $\hyperp$ with each normal of $V$.
By taking $\eps'$ to be sufficiently small with respect to $\eps$, we obtain
\begin{align*}
I(\pts,\hyperp) = mN^{\frac{2d}{d-1}-\eps'} &= \Theta\left(m^{\frac{3d^2-9d+2}{(d-2)(3d-1)}}\cdot \left(N^{\frac{2d}{d-1}-\eps'} m^{\frac{2d}{(d-2)(3d-1)}} \right)\right) \\
&\hspace{50mm}= \Omega\left(m^{\frac{3d^2-9d+2}{(d-2)(3d-1)}} n^{\frac{2d}{3d-1}-\eps}\right).
\end{align*}

Let $F \subset\RR^d$ be a 2-flat.
For a hyperplane $h\in \hyperp$ to contain $F$, the normal of $h$ must be orthogonal to two vectors that span $F$.
That is, the normal of the hyperplane is in a given linear $(d-2)$-dimensional subspace.
By the choice of $V$, we have that at most $t$ hyperplanes of $\hyperp$ contain any given 2-flat.
Since no line contains three points of $\pts$, we conclude that the incidence graph of $\pts\times\hyperp$ contains no copy of $K_{3,t+1}$.
\end{proof}
\vspace{2mm}

\noindent {\bf Theorem \ref{th:LowerOther}.} \emph{Consider positive integers $k,d,d'$ that satisfy $(d'-1)/d' > k/d$.
Then it is impossible to completely remove the dependency in $q_{d'-1,d'}$ from the incidence bound of Theorem \ref{th:IncReal} in the case of $k$-flats in $\RR^d$ and $s=2$.}
\begin{proof}
By Theorem \ref{th:LowerMain}(a), there exist a set $\pts$ of $m$ points and a set $\hyperp$ of $n$ hyperplanes, both in $\RR^{d'}$, such that the incidence graph of $\pts\times \hyperp$ contains no $K_{2,t+1}$ and $I(\pts,\hyperp) = \Omega\left(T_{d'-1,d'}(m,n) \cdot n^{-\eps}\right)$.

Let $F$ be an arbitrary $d'$-flat in $\RR^{d}$, and place the above point-flat configuration in $F$.
Then, replace each $(d'-1)$-flat $h$ in the configuration in $F$ with a generic $k$-flat in $\RR^d$ that contains $h$.
We may assume that the intersection of such a generic $k$-flat with $F$ is exactly $h$, so no new incidences are created.
Since $(d'-1)/d' > k/d$, Lemma \ref{le:ratio test} implies that $T_{d'-1,d'}(m,n) \cdot n^{-\eps}$ is asymptotically larger than $T_{k,d}(m,n)$ (when $\eps$ is sufficiently small).
That is, we have a configuration of $k$-flats in $\RR^d$, with asymptotically more than $T_{k,d}(m,n)$ incidences coming from a constant-degree variety of dimension $d'$ that has many $(d'-1)$-dimensional intersections with the $k$-flats.
\end{proof}

\section{Transverse varieties} \label{sec:Trans}

The goal of this section is to prove Theorem \ref{th:RdTrans}. Instead of directly proving this theorem, we prove the following more general result where the points are contained in a constant-degree variety $W\subset \RR^{d^*}$.
Theorem \ref{th:RdTrans} is immediately obtained by setting $W=\RR^d$ and $d^*=d$ in Theorem \ref{th:IncRealGen}.

\begin{theorem}
Let $k,d,d^*,s,t,D,D_W$ be positive integers with $s\geq 2$ and $k\le d/2$.
For any $\eps>0$, there exists a constant $c$ such that the following holds.
Let $\pts$ be a set of $m$ points on an irreducible variety $W\subseteq\RR^{d^*}$ of dimension $d$ and degree $D_W$.
Let $\vars$ be a set of $n$ irreducible varieties of degree at most $D$ in $\RR^{d^*}$, such that every $h\in \vars$ satisfies $\dim (h\cap W)\le k$.
Assume that the incidence graph of $\pts\times \vars$ contains no copy of $K_{s,t}$, and that the varieties of $\vars$ are transverse.
Moreover, for each pair $(k',d')\in \overline{R}_{k,d}$, every $d'$-dimensional variety of degree at most $c$ has a $k'$-dimensional intersection with at most $q_{k',d'}$ varieties of $\vars$.
Then
\begin{align*}
I^*(&\pts,\vars) \nonumber\\
&=O\Bigg( \sum_{( (k_0,d_0), \ldots, (k_u, d_u) ) \in \overline{S}_{k,d}} \hspace{-13mm}
m^{\frac{sk_u}{sd_u-d_u+k_u}+\eps}\Bigg(n^{\frac{d}{d-k}}q_{k_1,d_1}^{\frac{d_1}{d_1-k_1}-\frac{d}{d-k}}\cdots q_{k_u,d_u}^{\frac{d_u}{d_u-k_u}-\frac{d_{u-1}}{d_{u-1}-k_{u-1}}}\Bigg)^{\frac{(d_u-k_u)(s-1)}{sd_u-d_u+k_u}} \nonumber \\
&\hspace{118mm}+ m + n \Bigg).
\end{align*}
\end{theorem}
\begin{proof}
We use induction on $d$ to prove the more general statement.
In particular, we prove that
\begin{align}
I^*(\pts,\vars) \le \alpha_{1,k,d} \cdot \hspace{-13mm}\sum_{( (k_0,d_0), \ldots, (k_u, d_u) ) \in \overline{S}_{k,d}} \hspace{-13mm}
m^{\frac{sk_u}{sd_u-d_u+k_u}+\eps}\Bigg(n^{\frac{d}{d-k}}q_{k_1,d_1}^{\frac{d_1}{d_1-k_1}-\frac{d}{d-k}}\cdots &q_{k_u,d_u}^{\frac{d_u}{d_u-k_u}-\frac{d_{u-1}}{d_{u-1}-k_{u-1}}}\Bigg)^{\frac{(d_u-k_u)(s-1)}{sd_u-d_u+k_u}} \nonumber\\ &\hspace{17mm}+\alpha_{2,k,d}(m+n), \label{eq:FirstIndTrans}
\end{align}
where $\alpha_{1,k,d},\alpha_{2,k,d}$ are sufficiently large constants that depend on $d,k,s,t,D,D_W$, and $\eps$.
Out of these parameters we only place $d$ and $k$ in the subscript, to make parts of the proof easier to follow.
For the induction basis, we note that the claim is trivial when $d=2$.
Indeed, in this case $R_{1,2}$ is empty and $S_{1,2}$ contains only the empty sequence, so the required bound is the one in Theorem \ref{th:IncRealGen}.

To prove the induction step, we use a second induction on $m+n$.
For the base case of this second induction, when $m$ and $n$ are sufficiently small the result is obtained by choosing sufficiently large values for $\alpha_{1,k,d}$ and $\alpha_{2,k,d}$.
It remains to handle the induction step of the second induction.

By applying Theorem \ref{th:partition} with $U=W$, $d'=d$, and a sufficiently large constant $r$, we obtain a polynomial $f\in \RR[x_1,\ldots,x_{d^*}]$ of degree $O(r)$ such that every connected component of $W\setminus \vb(f)$ contains at most $m/r^{d}$ points of $\pts$.
Let $m' = |\pts \setminus \vb(f)|$.
Handling the incidences in the cells of the partition is almost identical to the proof of Theorem \ref{th:IncReal}.
The only change is replacing $R_{k,d}$ and $S_{k,d}$ with $\overline{R}_{k,d}$ and $\overline{S}_{k,d}$, respectively.
One subtle issue that arises is using the induction hypothesis with the new transversality restriction.
That is why the transversality is defined with respect to the original varieties of $\vars$ in $\RR^{d^*}$, and not with their intersections with $W$.
With this definition, replacing $W$ with a smaller variety cannot violate the transversality assumption.

We do not repeat the analysis of the cells from Theorem \ref{th:IncReal}, and refer the reader to the proof in Section \ref{sec:Rd}.
This analysis leads to
\begin{align}
I^*&(\pts\setminus \vb(f),\vars) \nonumber \\
&\le \frac{\alpha_{1,k,d}}{2} \Bigg(\sum_{( (k_0,d_0), \ldots, (k_u, d_u) ) \in \overline{S}_{k,d}}  \hspace{-10mm} m^{\frac{sk_u}{sd_u-d_u+k_u}+\eps}\Bigg(n^{\frac{d}{d-k}}q_{k_1,d_1}^{\frac{d_1}{d_1-k_1}-\frac{d}{d-k}}\hspace{-1mm} \cdots q_{k_u,d_u}^{\frac{d_u}{d_u-k_u}-\frac{d_{u-1}}{d_{u-1}-k_{u-1}}}\Bigg)^{\frac{(d_u-k_u)(s-1)}{sd_u-d_u+k_u}} \nonumber \\
&\hspace{118mm}+\alpha_{2,k,d}m'. \label{eq:incCellsTrans}
\end{align}

\parag{Incidences on the partition.}
It remains to bound the number of incidences with points that lie on $\vb(f)$.
Set $\pts_0= \pts \cap \vb(f)$ and $m_0=|\pts_0|=m-m'$.
We associate each point $p\in \pts_0$ with an arbitrary irreducible component of $\vb(f)\cap W$ that $p$ is incident to.
To bound the number of incidences with $\pts_0$, it suffices to separately consider each irreducible component of $\vb(f)\cap W$ and the points that are associated with it.

Let $U$ be an irreducible component of $\vb(f) \cap W$.
Let $\overline{\pts}$ be the set of points of $\pts$ associated with $U$, and set $\overline{m} = |\overline{\pts}|$.
By the definition of $f$, we have that $\dim U <d$.
By taking $c$ to be sufficiently large, we may assume that $\deg U\le c$.

Since the varieties of $\vars$ are transverse, at most one of these varieties can contain $U$.
We discard this variety, losing exactly $\overline{m}$ incidences with $\overline{\pts}$.
By Lemma \ref{le:BoundedNumComponents}, the number of incidences between $\overline{\pts}$ and varieties of $\vars$ that have a finite intersection with $U$ is also $O(n)$.
It remains to consider incidences with varieties $h\in \vars$ that satisfy $1\leq \dim(h\cap U)\leq \min\{k,d'-1\}$.
For $1\le k'\le \min\{k,d'-1\}$, let $\vars_{k'}$ denote the set of varieties of $\vars$ that have a $k'$-dimensional intersections with $U$.

For every $1\le k'\le \min\{d'/2,k\}$, we bound $I^*(\overline{\pts},\vars_{k'})$ by repeating the analysis in the proof of Theorem \ref{th:IncReal}.
Once again, the only change is replacing $R_{k,d}$ and $S_{k,d}$ with $\overline{R}_{k,d}$ and $\overline{S}_{k,d}$, respectively.
As in the proof of Theorem \ref{th:IncReal}, we obtain
\begin{align}
I^*(&\overline{\pts},\vars_{k'})  \nonumber \\
&= O_r\Bigg( \alpha_{1,k',d'}\cdot \hspace{-10mm}\sum_{( (k_0,d_0), \ldots, (k_u, d_u) ) \in \overline{S}_{k,d}} \hspace{-13mm}
m^{\frac{sk_u}{sd_u-d_u+k_u}+\eps}\Bigg(n^{\frac{d}{d-k}}q_{k_1,d_1}^{\frac{d_1}{d_1-k_1}-\frac{d}{d-k}}\cdots q_{k_u,d_u}^{\frac{d_u}{d_u-k_u}-\frac{d_{u-1}}{d_{u-1}-k_{u-1}}}\Bigg)^{\frac{(d_u-k_u)(s-1)}{sd_u-d_u+k_u}}  \nonumber \\
&\hspace{102mm}+\alpha_{2,k',d'}\left(\overline{m}+n\right)\Bigg). \label{eq:IncPartTrans}
\end{align}

We next consider $\vars_{k'}$ for a fixed $k'> d'/2$.
In this case we cannot repeat the proof of Theorem \ref{th:IncReal}, since we cannot apply the induction hypothesis.
We associate with every variety $h\in \vars_{k'}$ a subset $h_U \subseteq h\cap U$, and only consider incidences between $\overline{\pts}$ and $h_U$.
As a first step, we remove from $h_U$ every irreducible component of dimension at most $d'/2$, since such components can be included in the above argument.

We now show that every regular point of $U$ can be incident to at most one ``well-behaved'' element of $\vars_{k'}$.
Assume for contradiction that there exists a point $p$ that is a regular point of $U$, a regular point of two varieties $h,h'\in \vars_{k'}$, and a regular point of the corresponding $h_U$ and $h'_U$.
By the transversality assumption, $T_p h \cap T_p h'$ is the origin.
Since $T_p h_U \subset T_p h$ and $T_p h'_U \subset T_p h'$, we get that $T_p h_U \cap T_p h'_U$ is the origin.
Since both $T_p h_U$ and $T_p h'_U$ have dimension larger than $d'/2$, together they span a space of dimension larger than $d'$.
This is impossible, since both of these tangent flats should be contained in the $d'$-dimensional flat $T_p U$.
Thus, for every regular point $p\in U$, at most one $h\in \vars_{k'}$ satisfies that $p$ is a regular point of both $h$ and $h_U$.
Such containments lead to at most $\overline{m}$ incidences.

While we may ignore incidences with the singular points of a variety $h\in \vars_{k'}$, we still need to consider regular points of $h$ that are singular points of $h_U$.
For this purpose, we replace every irreducible component of $h_U$ with the set of singular points of this component.
By Theorem \ref{th:Singular}, the revised $h_U$ is of dimension at most $k'-1$.
We repeat the above process to handle incidences with the revised elements $h_U$.
That is, we rely on the proof of Theorem \ref{th:IncReal} to handle components of dimension at most $d'/2$, and observe that every regular point of $U$ is incident to at most one of the ``well-behaved'' elements of dimension larger than $d'/2$.
We repeat this process again and again, each time replacing every variety $h_U$ with the set of its singular points, until we remain only with varieties of dimension at most $d'/2$.
This process ends after at most $d'/2-1$ steps.

It remains to handle incidences with singular points of $U$.
We do this by separately repeating the above process for every irreducible component $\overline{U}$ of $U_\text{sing}$.
That is, we rely on the proof of Theorem \ref{th:IncReal} to handle low-dimensional components in $\vars_{k'}$, and observe that every regular point of $\overline{U}$ is a regular point of at most one of the higher-dimensional intersections.
We then need to repeat the proof for the singular set of $\overline{U}$, and so on.
By Theorem \ref{th:Singular}, at each step the dimension of the irreducible components decreases.
Thus, this process will require at most $d'$ steps.
At each step we also need to repeat the process described in the previous paragraph.

At every step of the above process, the number of irreducible components that are handled separately is multiplied by a factor of $O_{d,r}(1)$.
Thus,  the total number of components is $O_{d,r}(1)$.
Since in each step the number of additional incidences is either $O(\overline{m})$ or \eqref{eq:IncPartTrans}, we again obtain the bound \eqref{eq:IncPartTrans}.

By Lemma \ref{le:BoundedNumComponents}, the variety $W \cap \vb(f)$ consist of $O_r(1)$ irreducible components.
Set $\alpha_{1,k,d}' = \alpha_{1,k,d-1} +\alpha_{1,k-1,d}$ and $\alpha_{2,k,d}' = \alpha_{2,k,d-1} +\alpha_{2,k-1,d}$.
Summing \eqref{eq:IncPartTrans} over each of those implies
\begin{align*}
I^*(&\pts_0,\vars)  \\
&= O_{r,d}\Bigg( \alpha'_{1,k,d}\cdot \hspace{-10mm}\sum_{( (k_0,d_0), \ldots, (k_u, d_u) ) \in \overline{S}_{k,d}} \hspace{-13mm}
m^{\frac{sk_u}{sd_u-d_u+k_u}+\eps}\Bigg(n^{\frac{d}{d-k}}q_{k_1,d_1}^{\frac{d_1}{d_1-k_1}-\frac{d}{d-k}}\cdots q_{k_u,d_u}^{\frac{d_u}{d_u-k_u}-\frac{d_{u-1}}{d_{u-1}-k_{u-1}}}\Bigg)^{\frac{(d_u-k_u)(s-1)}{sd_u-d_u+k_u}}  \\
&\hspace{112mm}+\alpha'_{2,k,d}\left(m_0+n\right)\Bigg).
\end{align*}

By taking $\alpha_{1,d,k}$ to be sufficiently large with respect to $r,d,$ and $\alpha'_{1,d,k}$, and taking $\alpha_{2,d,k}$ to be sufficiently large with respect to $r,d,$ and $\alpha'_{2,d,k}$, we conclude that
\begin{align*}
I^*(&\pts_0,\vars)  \\
&= \frac{\alpha_{1,k,d}}{2}\cdot \hspace{-10mm}\sum_{( (k_0,d_0), \ldots, (k_u, d_u) ) \in \overline{S}_{k,d}} \hspace{-13mm}
m^{\frac{sk_u}{sd_u-d_u+k_u}+\eps}\Bigg(n^{\frac{d}{d-k}}q_{k_1,d_1}^{\frac{d_1}{d_1-k_1}-\frac{d}{d-k}}\cdots q_{k_u,d_u}^{\frac{d_u}{d_u-k_u}-\frac{d_{u-1}}{d_{u-1}-k_{u-1}}}\Bigg)^{\frac{(d_u-k_u)(s-1)}{sd_u-d_u+k_u}}  \\
&\hspace{112mm}+\alpha_{2,k,d}\left(m_0+n\right).
\end{align*}

Combining this with \eqref{eq:incCellsTrans} completes the two induction steps and the proof of the theorem.
\end{proof}

\section{Hyperplanes in $\CC^d$} \label{sec:Complex}

In this section we prove Theorem \ref{th:Cd}.
We write $z\in \CC$ as $z_x+iz_y$, where $z_x,z_y\in \RR$.
Sometimes we have $z_j\in \CC$, and then we write $z_j = z_{j,x}+iz_{j,y}$.
The polynomial partitioning technique does not work in complex spaces, since removing a variety cannot disconnect a complex space.
To overcome this issue, one often thinks of $\CC^d$ as $\RR^{2d}$.
We will rely on this approach, considering the map $\phi: \CC^d \to \RR^{2d}$ defined by
\[ \phi\left(z_1,\ldots,z_d\right)=\left(z_{1,x},z_{1,y},z_{2,x},z_{2,y},\ldots,z_{d,x},z_{d,y}\right). \]

Consider a complex hyperplane $h$ in $\CC^d$.
By definition, $h$ can be defined using a linear equation $a_1z_1+a_2z_2+\cdots+a_dz_d=b$, where $a_1,\ldots,a_d,b\in \CC$ and $z_1,\ldots,z_d$ are the complex coordinates of $\CC^d$.
Then, $\phi(h)$ is the variety in $\RR^{2d}$ defined by
\begin{align*}
a_{1,x}z_{1,x} -a_{1,y}z_{1,y}+\cdots +a_{d,x}z_{d,x} -a_{d,y}z_{d,y} = b_1, \\
a_{1,x}z_{1,y} + a_{1,y}z_{1,x} + \cdots + a_{d,x}z_{d,y} + a_{d,y}z_{d,x} = d_2.
\end{align*}

It is not difficult to verify that this system defines a $(2d-2)$-flat in $\RR^{2d}$, unless $a_1=\cdots=a_d=0$.
A point $p\in \CC^d$ is incident to $h$ if and only if the point $\phi(p)$ is incident to the flat $\phi(h)$.
Thus, we can reduce the problem of incidences with planes in $\CC^d$ to incidences with $(2d-2)$-flats in $\RR^{2d}$.

Since the proof of Theorem \ref{th:Cd} involves a long, technical, and straightforward case analysis, we begin by proving the special case of incidences with planes in $\CC^3$.
In this case, the problem is reduced to an incidence problem between points and 4-flats in $\RR^6$.
After presenting the full details of the proof of this special case, we explain how to extend this proof to incidences in $\CC^d$.
We explain the steps of the general proof, but do not write the (standard) full technical details of each case in this analysis.

In an incidence problem with 4-flats in $\RR^6$, the dimension ratio is 2/3.
According to Lemma \ref{le:ratio test}, we should worry about the dimension ratios 4/5 and 3/4;
that is, worry about 4-flats contained in five-dimensional components of the partition, and about 4-flats that have a three-dimensional intersection with four-dimensional components of the partition.
To handle the former case, we note that 4-flats that originate from planes in $\CC^3$ and are contained in a five-dimensional variety $U$ can intersect only in singular points of $U$.
To handle the latter case, we use the following lemma.

\begin{lemma} \label{le:compFlatsC3}
Let $U$ be a four-dimensional variety in $\RR^6$, and let $p$ be a regular point of $U$.
Let $h_1,\ldots,h_u$ be $u\ge 2$ planes in $\CC^3$ such that each of the 4-flats $\phi(h_1),\ldots,\phi(h_u)$ has a three-dimensional intersection with $U$.
Moreover, assume that $p$ is a regular point of $U \cap \phi(h_j)$ for every $1\le j\le u$.
Then $h_1 \cap \cdots \cap h_u$ is a complex line.
\end{lemma}
\begin{proof}
The claim is straightforward when $u=2$, so we assume that $u\ge 3$.
For $1\le j \le3$, let $F_j =T_p(U \cap \phi(h_j))$.
Note that each $F_j$ is a 3-flat contained in $T_p U$.
Since the 4-flat $T_p U$ contains the 3-flats $F_1,F_2,F_3$ and $p\in F_1\cap F_2 \cap F_3$, we obtain that $\dim (F_1 \cap F_2 \cap F_3 \cap T_p U) \ge 1$.
Since $F_j \subset \phi(h_j)$ for every $1\le j\le 3$, we obtain that $\dim (\phi(h_1)\cap \phi(h_2) \cap \phi(h_3)) \ge 1$.
This in turn implies that $\dim (h_1\cap h_2 \cap h_3) \ge 1$, so the three complex planes intersect in a complex line $\ell\subset \CC^3$.

The above completes the proof for the case of $u=3$.
We now assume that $u>3$ and let $3< j \le u$.
Repeating the above argument for $h_1,h_2,h_j$ implies that $h_1\cap h_2\cap h_j$ is a complex line $\ell_j \subset\CC^3$.
Since $h_1\cap h_2$ is a single line, we obtain that $\ell_j = \ell$.
Since this holds for every $3< j \le u$, we conclude that $h_1\cap \ldots \cap h_u = \ell$.
\end{proof}

We are now ready to prove the special case of incidences with planes in $\CC^3$.

\begin{theorem} \label{th:C3}
Let $s$ and $t$ be positive integers with $s\geq 2$.
Let $\pts$ be a set of $m$ points and let $\vars$ be a set of $n$ planes, both in $\CC^3$.
Assume that the incidence graph of $\pts\times \vars$ contains no copy of $K_{s,t}$.
Then for any $\eps>0$ we have
\begin{align*}
I(\pts,\vars)=O\left( m^{\frac{2s}{3s-1}+\eps}n^{\frac{3s-3}{3s-1}}+ m + n \right).
\end{align*}
\end{theorem}
\begin{proof}
As explained in the beginning of this section, $I(\pts,\vars) = I(\phi(\pts),\phi(\vars))$, so it suffices to bound the latter.
Note that the incidence graph of $\phi(\pts)\times\phi(\vars)$ also contains no copy of $K_{s,t}$.
Abusing the notation, in the rest of the proof we write $\pts$ instead of $\phi(\pts)$ and $\vars$ instead of $\phi(\vars)$.

To prove the theorem, we prove by induction on $m+n$ that
\begin{equation} \label{eq:FirstIndComp}
I(\pts,\vars) \le \alpha_{1} m^{\frac{2s}{3s-1}+\eps}n^{\frac{3s-3}{3s-1}} +\alpha_{2}(m+n),
\end{equation}
where $\alpha_{1},\alpha_{2}$ are sufficiently large constants that depend on $s,t,$ and $\eps$.
For the induction basis, the case where $m$ and $n$ are sufficiently small can be handled by choosing sufficiently large values of $\alpha_{1}$ and $\alpha_{2}$.
Throughout the proof, the hidden constants in the $O(\cdot)$-notations may also depend on $s,t,$ and $\eps$.
For brevity we write $O(\cdot)$ instead of $O_{s,t,\eps}(\cdot)$.

Since the incidence graph contains no copy of $K_{s,t}$, Lemma \ref{le:KSTgen} implies $I(\pts,\vars) = O(mn^{1-1/s}+n)$.
When $m=O(n^{1/s})$, this implies $I(\pts,\vars) = O(n)$. We may thus assume that
\begin{equation} \label{eq:kstAssumComp}
n=O(m^s).
\end{equation}

\parag{Partitioning the space.}
By Theorem \ref{th:partition} with $U=\RR^6$ and a constant $r$, we obtain a polynomial $f\in \RR[x_1,\ldots,x_6]$ of degree $O(r)$ such that every connected component of $\RR^6\setminus \vb(f)$ contains at most $m/r^6$ points of $\pts$.
The asymptotic relations between the various constants in the proof are
\[ 2^{1/\eps} \ll r \ll \overline{r} \ll \overline{\alpha}_2 \ll \overline{\alpha}_1 \ll \alpha_{2} \ll \alpha_{1}.\]

Denote the cells of the partition as $C_1, \ldots, C_v$.
By Theorem \ref{th:BaroneBasu}, we have that $v=O(r^6)$.
For each $1\le j \le v$, denote by $\vars_j$ the set of varieties of $\vars$ that intersect $C_j$, and set $\pts_j = C_j \cap \pts$.
We also set $m_j=|\pts_j|$, $m' = \sum_{j=1}^v m_j$, and $n_j=|\vars_j|$.
Note that $m_j\le m/r^6$ for every $1\le j \le v$.
By Theorem \ref{th:BaroneBasu}, every variety of $\vars$ intersects $O(r^4)$ cells of $\RR^6 \setminus \vb(f)$.
Therefore, $\sum_{j=1}^v n_j = O(nr^4)$.
H\"older's inequality implies
\begin{align*}
\sum_{j=1}^v n_j^{\frac{3s-3}{3s-1}}&\le \left(\sum_{j=1}^v n_j\right)^{\frac{3s-3}{3s-1}} \left(\sum_{j=1}^v 1\right)^{\frac{2}{3s-1}} = O\left(\left(nr^4\right)^{\frac{3s-3}{3s-1}} r^{\frac{12}{3s-1}}\right) = O\left(n^{\frac{3s-3}{3s-1}} r^{\frac{12s}{3s-1}}\right).
\end{align*}

Combining the above with the induction hypothesis implies
\begin{align*}
\sum_{j=1}^v I(\pts_j,\vars_j) &\le \sum_{j=1}^v \left(\alpha_{1} m_j^{\frac{2s}{3s-1}+\eps} n_j^{\frac{3s-3}{3s-1}} +\alpha_{2}(m_j+n_j)\right) \\
&\le  \alpha_{1} \frac{m^{\frac{2s}{3s-1}+\eps}}{r^{\frac{12s}{3s-1}+6\eps}}\sum_{j=1}^v n_j^{\frac{3s-3}{3s-1}} + \alpha_{2}\left(m'+O\left(nr^4\right)\right) \\[2mm]
&\le  \alpha_{1} \frac{m^{\frac{2s}{3s-1}+\eps}}{r^{6\eps}} n^{\frac{3s-3}{3s-1}} +\alpha_{2}\left(m'+O\left(nr^4\right)\right).
\end{align*}

By \eqref{eq:kstAssumComp} we have $n=O\left(m^{\frac{2s}{3s-1}}n^{\frac{3s-3}{3s-1}}\right)$.
Thus, when $\alpha_{1}$ is sufficiently large with respect to $r$ and $\alpha_{2}$, we get
\begin{align*}
\sum_{j=1}^v& I(\pts_j,\vars_j) =O\left(\alpha_{1}  \frac{m^{\frac{2s}{3s-1}+\eps}}{r^{6\eps}}n^{\frac{3s-3}{3s-1}}\right) +\alpha_{2}m'.
\end{align*}

When $r$ is sufficiently large with respect to $\eps$ and to the constant hidden in the $O(\cdot)$-notation, we have
\begin{equation} \label{eq:incCellsComp}
\sum_{j=1}^v I(\pts_j,\vars_j)  \le \frac{\alpha_{1}}{2}  m^{\frac{2s}{3s-1}+\eps}n^{\frac{3s-3}{3s-1}} +\alpha_{2}m'.
\end{equation}

\parag{Incidences on the partition} It remains to bound the number of incidences with points that lie on $\vb(f)$.
Set $\pts_0 = \pts \cap \vb(f)$ and $m_0 = |\pts_0|$.
We associate each point $p\in \pts_0$ with an arbitrary component of $\vb(f)$ that $p$ is incident to.
To bound the number of incidences with $\pts_0$, it suffices to separately consider each irreducible component of $\vb(f)$ and the points that are associated with it.

Let $U$ be an irreducible component of $\vb(f)$, let $\overline{\pts}$ denote the set of points that are associated with $U$, and set $\overline{m} = |\overline{\pts}|$.
Since $\deg f = O(r)$, the degree of $U$ is also $O(r)$.
We divide the analysis of the incidence with $\overline{\pts}$ according to $\dim U$.

We first assume that $\dim U \le 3$.
Since the incidence graph contains no copy of $K_{s,t}$, either $U$ is contained in fewer than $t$ flats of $\vars$ or $\overline{m}< s$.
In the former case, flats that contain $U$ form $O(\overline{m})$ incidences with $\overline{\pts}$, and in the latter they form $O(n)$ incidences.
It remains to consider flats of $\vars$ that intersect $U$ in a variety of dimension at most two.
This is an incidence problem between $\overline{\pts}$ and at most $n$ varieties of dimension at most two.
Let $\overline{\vars}_2$ be the set of these varieties.

Let $h$ be a generic 3-flat in $\RR^6$, and let $\pi_3(\cdot): \RR^6 \to \RR^3$ be the projection on $h$.
Since $h$ is a generic 3-flat, we may assume that no two points of $\overline{\pts}$ are projected to the same point of $h$.
Similarly, we may assume that the projections of $\overline{\pts}$ and $\overline{\vars}_2$ on $h$ do not lead to any new incidences.
That is, $I(\overline{\pts},\overline{\vars}_2) = I(\pi_3(\overline{\pts}),\pi_3(\overline{\vars}_2))$ and the incidence graph of $\pi_3(\overline{\pts})\times\pi_3(\overline{\vars}_2)$ contains no copy of $K_{s,t}$.
Theorem \ref{th:UpperBounds}(b) implies
\begin{equation*}
I(\overline{\pts},\overline{\vars_2}) = I(\pi_3(\overline{\pts}),\pi_3(\overline{\vars}_2)) = O\left(m^{\frac{2s}{3s-1}+\eps}n^{\frac{3s-3}{3s-1}}+\overline{m}+n \right).
\end{equation*}

We conclude that when $\dim U \le 3$, we have
\begin{equation} \label{eq:CompCases}
I(\overline{\pts},\vars) = O_r\left(m^{\frac{2s}{3s-1}+\eps}n^{\frac{3s-3}{3s-1}}+\overline{m}+n \right).
\end{equation}

\parag{The four-dimensional case.} We now assume that $\dim U = 4$.
In this case, at most one 4-flat of $\vars$ can contain $U$, and this flat forms $\overline{m}$ incidences with $\overline{\pts}$.
Every other 4-flat of $\vars$ intersects $U$ in a variety of dimension at most three.

By Theorem \ref{th:Singular}, the singular set $U_\text{sing}$ is a variety of dimension at most three and degree $O_r(1)$.
Thus, we can bound the number of incidences with points of $U_\text{sing}$ in the same way we handled the case of $\dim U \le 3$.
In particular, $U_\text{sing}$ consists of $O_r(1)$ components, and each should be handled separately as in the case of $\dim U \le 3$.
Similarly, incidences with 4-flats of $\vars$ that intersect $U$ in a variety of dimension at most two can be handled using the projection argument involving $\pi_3(\cdot)$.
If $h\in \vars$ has a three-dimensional intersection with $U$, we can also include any lower-dimensional components of this intersection in the above projection argument.
It remains to deal with 4-flats of $\vars$ that intersect $U$ in a three-dimensional variety, and only with the three-dimensional components of this intersection.

Let $h$ be a 4-flat of $\vars$ that has a three-dimensional intersection with $U$.
By Theorem \ref{th:Singular}, the singular set of $h\cap U$ is a variety of dimension at most two and degree $O_{r}(1)$, which can also be included in the projection argument above.
It remains to deal with the regular points of the three-dimensional components of $h\cap U$.
Let $\overline{\vars}_3$ be the set of the remaining portions of the intersections between $U$ and $\vars$.
That is, for each intersection we only include the three-dimensional components, and consider only incidences with their regular points.

It remains to handle incidences between points of $\overline{\pts}$ that are regular points of $U$ and regular points of elements of $\overline{\vars}_3$.
By Lemma \ref{le:compFlatsC3}, for each such point $p$ there exists a complex line $\ell_p\subset \CC^3$ that is contained in each of the complex planes whose corresponding elements in $\overline{\vars}_3$ have $p$ as a regular point.
Let $F_p = \phi(\ell_p)$.
Let $\flats$ be the set of these 2-flats (one 2-flat for each point of $\overline{\pts}$ that is a regular point of $U$).
To bound the number of incidences between the remaining points and the regular points of the elements of $\overline{\vars}_3$, it suffices to derive an upper bound for the number of containments between the 2-flats of $\flats$ and the 4-flats of $\vars$.

Let $H$ be a generic 4-flat in $\RR^6$.
We may assume that $H$ has a two-dimensional intersection with every 4-flat of $\vars$ and intersects each 2-flat of $\flats$ in a point.
Thus, inside of $H$ the above containment problem becomes an incidence problem between points and 2-flats.
By projecting this incidence problem on a generic 3-flat and applying Theorem \ref{th:UpperBounds}(b), we obtain the bound
\begin{equation} \label{eq:ContainmentIncRed}
O\left(m^{\frac{2s}{3s-1}+\eps}n^{\frac{3s-3}{3s-1}}+\overline{m}+n \right).
\end{equation}

There is a delicate issue with the above argument: It is possible that several regular points of $U$ correspond to the same 2-flat, and then a single containment between a 2-flat and a 4-flat corresponds to several incidences in $\RR^6$.
Recalling that the original incidence graph contains no $K_{s,t}$, we obtain that any 2-flat that corresponds to at least $s$ points is contained in at most $t-1$ of the 4-flats. 
The number of incidences coming from such containments is $O(m)$.
After discarding the 2-flats leading to these incidences, we are left with 2-flats that correspond to at most $s-1$ points of $U$. 
We may thus apply Theorem \ref{th:UpperBounds}(b) as before, multiply by $s-1$, and add the $O(m)$ discarded incidences.
This process leads to the same bound as in \eqref{eq:ContainmentIncRed}.

By combining all of the above cases, we conclude that when $\dim U = 4$, the bound \eqref{eq:CompCases} holds.

\parag{The five-dimensional case.} It remains to consider the case where $\dim U = 5$.
By Theorem \ref{th:Singular}, the singular set $U_\text{sing}$ is a variety of dimension at most four and degree $O_r(1)$.
We can thus bound the number of incidences with points of $U_\text{sing}\cap \overline{\pts}$ in the same way we handled the cases of $\dim U \le 4$.
In particular, $U_\text{sing}$ consists of $O_r(1)$ components, and each should be handled separately as in the cases of $\dim U \le 4$.
Similarly, incidences with 4-flats of $\vars$ that intersect $U$ in a variety of dimension at most two can be handled using the above projection argument involving $\pi_3(\cdot)$.

Consider two complex planes $h_1,h_2\subset \CC^3$ that intersect at a point $p \in \CC^3$.
After translating $\CC^3$ so that $p$ becomes the origin, we get that the linear subspaces $h_1$ and $h_2$ span $\CC^3$.
This in turn means that $\phi(h_1)$ and $\phi(h_2)$ span $\RR^6$.
If $\phi(h_1),\phi(h_2)\in U$, then $\phi(p)$ must be a singular point of $U$.
Indeed, if $p$ was a regular point of $U$ then we would have $\phi(h_1),\phi(h_2) \subset T_p U$.
This is impossible since at a regular point $p$, the tangent $T_p U$ is a hyperplane while $\phi(h_1),\phi(h_2)$ span $\RR^6$.
We conclude that for every $p\in \overline{\pts}$ that is a regular point of $U$, at most one 4-flat of $\vars$ is contained in $U$ and incident to $p$.
There are at most $\overline{m}$ such incidences.

It remains to consider 4-flats of $\vars$ that have a three-dimensional intersection with $U$.
Let $\overline{\vars}_3$ denote the set of such three-dimensional intersections.
We clearly have $|\overline{\vars}_3|\le n$.
We derive an upper bound for $I(\overline{\pts},\overline{\vars}_3)$ using a second induction.
In particular, we prove by induction on $\overline{m}+n$ that
\begin{equation} \label{eq:SecondIndComp}
I(\overline{\pts},\overline{\vars}_3) \le \overline{\alpha}_{1} \overline{m}^{\frac{2s}{3s-1}+\eps}n^{\frac{3s-3}{3s-1}} +\overline{\alpha}_{2}(\overline{m}+n),
\end{equation}
where $\overline{\alpha}_{1},\overline{\alpha}_{2}$ are sufficiently large constants that depend on $s,t,r,$ and $\eps$.
For the induction basis, the case where $\overline{m}$ and $n$ are sufficiently small can be handled by choosing sufficiently large values of $\overline{\alpha}_{1}$ and $\overline{\alpha}_{2}$.

For the induction step, we use a second partitioning polynomial.
By Theorem \ref{th:partition} with $U$ and a sufficiently large constant $\overline{r}$, there exists a polynomial $\overline{f}\in \RR[x_1,\ldots,x_6]$ of degree $O(\overline{r})$ that does not vanish identically on $U$, such that every connected component of $U\setminus \vb(\overline{f})$ contains at most $\overline{m}/\overline{r}^5$ points of $\overline{\pts}$.

Denote the cells of the partition as $\overline{C}_1, \ldots, \overline{C}_{\overline{v}}$.
By Theorem \ref{th:BaroneBasu}, we have that $\overline{v}=O(\overline{r}^5)$.
For each $1\le j \le \overline{v}$, denote by $\overline{\vars}_j$ the set of varieties of $\overline{\vars}_3$ that intersect $\overline{C}_j$, and set $\overline{\pts}_j = \overline{C}_j \cap \overline{\pts}$.
We also set $\overline{m}_j=|\overline{\pts}_j|$, $\overline{m}' = \sum_{j=1}^{\overline{v}} \overline{m}_j$, and $\overline{n}_j=|\overline{\vars}_j|$.
Note that $\overline{m}_j\le \overline{m}/\overline{r}^5$ for every $1\le j \le \overline{v}$.
By Theorem \ref{th:BaroneBasu}, every variety of $\overline{\vars}_3$ intersects $O(\overline{r}^3)$ cells of $U \setminus \vb(\overline{f})$.
Therefore, $\sum_{j=1}^{\overline{v}} \overline{n}_j = O(n\overline{r}^3)$.
H\"older's inequality implies
\begin{align*}
\sum_{j=1}^{\overline{v}} \overline{n}_j^{\frac{3s-3}{3s-1}}&\le \left(\sum_{j=1}^{\overline{v}} \overline{n}_j\right)^{\frac{3s-3}{3s-1}} \left(\sum_{j=1}^{\overline{v}} 1\right)^{\frac{2}{3s-1}} = O\left(\left(n\overline{r}^3\right)^{\frac{3s-3}{3s-1}} \overline{r}^{\frac{10}{3s-1}}\right) = O\left(n^{\frac{3s-3}{3s-1}} \overline{r}^{\frac{9s+1}{3s-1}}\right).
\end{align*}

Combining the above with the induction hypothesis implies
\begin{align*}
\sum_{j=1}^{\overline{v}} I(\overline{\pts}_j,\overline{\vars}_j) &\le \sum_{j=1}^{\overline{v}} \left(\overline{\alpha}_{1} \overline{m}_j^{\frac{2s}{3s-1}+\eps} \overline{n}_j^{\frac{3s-1}{3s-3}} +\overline{\alpha}_{2}(\overline{m}_j+\overline{n}_j)\right) \\
&\le \overline{\alpha}_{1} \frac{m^{\frac{2s}{3s-1}+\eps}}{\overline{r}^{\frac{10s}{3s-1}+5\eps}}\sum_{j=1}^v \overline{n}_j^{\frac{3s-3}{3s-1}} + \overline{\alpha}_{2}\left(\overline{m}'+O\left(n\overline{r}^3\right)\right) \\[2mm]
&\le  \overline{\alpha}_{1} \frac{m^{\frac{2s}{3s-1}+\eps}}{\overline{r}^{5\eps}} n^{\frac{3s-3}{3s-1}} +\overline{\alpha}_{2}\left(\overline{m}'+O\left(n\overline{r}^3\right)\right).
\end{align*}

We recall that $n=O\left(m^{\frac{2s}{3s-1}}n^{\frac{3s-3}{3s-1}}\right)$.
When $\overline{\alpha}_{1}$ is sufficiently large with respect to $\overline{r}$ and $\overline{\alpha}_{2}$, we get
\begin{align*}
\sum_{j=1}^{\overline{v}} & I(\overline{\pts}_j,\overline{\vars}_j) =O\left(\overline{\alpha}_{1}  \frac{m^{\frac{2s}{3s-1}+\eps}}{\overline{r}^{5\eps}}n^{\frac{3s-3}{3s-1}}\right) +\overline{\alpha}_{2}\overline{m}'.
\end{align*}

When $\overline{r}$ is sufficiently large with respect to $\eps$ and to the constant hidden in the $O(\cdot)$-notation, we have
\begin{equation} \label{eq:SecondIndCompCells}
\sum_{j=1}^{\overline{v}} I(\overline{\pts}_j,\overline{\vars}_j)  \le \frac{\overline{\alpha}_{1}}{2}  m^{\frac{2s}{3s-1}+\eps}n^{\frac{3s-3}{3s-1}} +\overline{\alpha}_{2}\overline{m}'.
\end{equation}

It remains to bound the number of incidences with points of $\vb(\overline{f})\cap \overline{\pts}$.
Since $\vb(\overline{f})$ is a variety of dimension at most four and degree $O_{\overline{r}}(1)$, we can simply repeat the analysis of the cases of $\dim U \le 4$.
We get the upper bound in \eqref{eq:CompCases} for this number of incidences (with $r$ replaced by the larger constant $\overline{r}$).
By taking $\overline{\alpha}_1$ and $\overline{\alpha}_2$ to be sufficiently large with respect to $\overline{r}$, we obtain
\[ I(\overline{\pts}\cap \vb(\overline{f}),\overline{\vars}_3)  \le \frac{\overline{\alpha}_{1}}{2}  m^{\frac{2s}{3s-1}+\eps}n^{\frac{3s-3}{3s-1}} +\overline{\alpha}_{2}(\overline{m}-\overline{m}'+n). \]

Combining this with \eqref{eq:SecondIndCompCells} completes the proof of the second induction step.
We conclude that when $\dim U = 5$, the bound \eqref{eq:CompCases} also holds.

\parag{Completing the proof.}
By going over the various cases above, we note that in every case we have \eqref{eq:CompCases}.
By Lemma \ref{le:BoundedNumComponents}, the partition $\vb(f)$ consists of $O_r(1)$ irreducible components.
Summing up \eqref{eq:CompCases} over every irreducible component of $\vb(f)$ leads to
\begin{equation*}
I(\pts_0,\vars) = O_{r,\overline{r}}\left(m^{\frac{2s}{3s-1}+\eps}n^{\frac{3s-3}{3s-1}}+m_0+n \right)
\end{equation*}

Taking $\alpha_1$ and $\alpha_2$ to be sufficiently large, we obtain
\begin{equation*}
I(\pts_0,\vars) \le \frac{\alpha_1}{2} m^{\frac{2s}{3s-1}+\eps}n^{\frac{3s-3}{3s-1}}+\alpha_2( m_0+n).
\end{equation*}

Combining this bound with \eqref{eq:incCellsComp} completes the first induction step and the proof of the theorem.
\end{proof}

To generalize Theorem \ref{th:C3} to incidences with hyperplanes in $\CC^d$, we first generalize Lemma \ref{le:compFlatsC3}.

\begin{lemma}\label{le:compFlats}
Let $U$ be an $e$-dimensional variety in $\RR^{2d}$, where $d < e < 2d$.
Let $h_1,\dots, h_u$ be $u$ distinct hyperplanes in $\CC^d$ such that  each of the corresponding $(2d-2)$-flats $\phi(h_j)$ has an $(e-1)$-dimensional intersection with $U$.
Let $p\in \RR^{2d}$ be a regular point of each of these intersections with $U$, and also a regular point of $U$.
Then $\dim (\cap_{j=1}^u h_j)\geq 1$.
\end{lemma}
\begin{proof}
When $u<d$ the statement of the lemma is trivial, so assume that $u\ge d$.
For $1\leq j\leq u$, set $F_j=T_p(U\cap \phi(h_j))$.
Each $F_j$ is an $(e-1)$-flat inside $T_p(U)$.
We consider $T_p(U)$ as $\RR^e$ with $p$ as the origin, and then each $F_j$ is a linear subspace.
Recall that any linear subspaces $A$ and $B$ satisfy $\dim(A\cap B)\geq \dim A+\dim B-\dim (A\cup B)$.
In particular, if $B$ has co-dimension 1 then $\dim(A\cap B)\geq \dim A-1$.
Applying this repeatedly with $1\leq j_1<\cdots<j_d\leq u$ leads to $\dim (F_{j_1}\cap \cdots \cap F_{j_d})\geq (e-1)-1-\cdots-1\ge 1$ (since $e>d$).
This implies any $d$ hyperplanes from $h_1,\cdots, h_u$ contain a common line in $\CC^d$.

Let $d_0$ be the smallest non-negative integer such that there exist hyperplanes $h_{j_1},\dots, h_{j_{d-d_0}}$ whose intersection is of dimension $d_0$.
Since every two of the hyperplanes have an intersection of dimension $d-2$, we have that $d_0 \le d-2$.
Since every $d$ of the hyperplanes contain a common line, we have $1\le d_0 \le d-2$.
We consider hyperplanes $h_{j_1},\cdots, h_{j_{d-d_0}}$ such that the intersection $I=h_{j_1}\cap \cdots\cap h_{j_{d-d_0}}$ is a $d_0$-flat.
By the definition of $d_0$, for a hyperplane $h \notin \{h_1,\ldots,h_{j_{d-d_0}}\}$ we have $\dim (h\cap I)\ge d_0$.
This implies that $h$ contains $I$.
Since this holds for every $h$, we conclude that $\dim (\cap_{j=1}^u h_j)\geq d_0\geq 1$.
\end{proof}

We now discuss the proof of Theorem \ref{th:Cd}, and first recall the statement of this theorem.
\vspace{2mm}

\noindent {\bf Theorem \ref{th:Cd}.}
\emph{Let $s,t \ge 2$ be integers.
Let $\pts$ be a set of $m$ points and let $\vars$ be a set of $n$ hyperplanes, both in $\CC^d$.
Assume that the incidence graph of $\pts\times \vars$ contains no copy of $K_{s,t}$.
Then for any $\eps>0$ we have}
\begin{align*}
I(\pts,\vars)=O\left( m^{\frac{(d-1)s}{ds-1}+\eps}n^{\frac{ds-d}{ds-1}}+ m + n \right).
\end{align*}
\begin{proof}[Proof sketch.]
The proof goes along the same lines as the proofs of Theorems \ref{th:C3} and \ref{th:IncReal}, but with more cases to handle.
As in the case of Theorem \ref{th:IncReal}, the proof is by induction on $d$.
For the induction basis, the case of $d=3$ is Theorem \ref{th:C3} and the case of $d=2$ can be found in \cite{ST12}.
We prove the induction step using a second induction on $m+n$.
The second induction basis is immediate by taking sufficiently large constants in the $O(\cdot)$-notation.

To prove the second induction step, we move from $\CC^d$ to $\RR^{2d}$ using the map $\phi(\cdot)$, as described in the beginning of the section.
We then take a constant-degree partitioning polynomial $f\in \RR[x_1,\ldots,x_{2d}]$.
Handling the incidences in the cells is straightforward, by applying the induction hypothesis separately in each cell and using H\"older's inequality (see the proofs of Theorems \ref{th:C3} and \ref{th:IncReal}).
To close the induction step, in remains to bound the number of incidences on the partition $\vb(f)$.

As in the previous proofs, we separately consider each irreducible component of $\vb(f)$.
Let $U$ be such a component of dimension $d'$.
For each $0\le k <d'$, we separately bound the number of incidences between points of $\pts\cap U$ and hyperplanes of $\phi(\vars)$ that have a $k$-dimensional intersection with $U$.
We denote as $\vars_k$ the set of these $k$-dimensional varieties in $U$.

When $k<d$, we simply project the incidence problem onto a generic $d$-flat and bound the number of incidences using Theorem \ref{th:UpperBounds}(b).
That is, we project the point set $\pts\cap U$ and the set of varieties $\vars_k$.
By Lemma \ref{le:ratio test}, the resulting incidence bounds are subsumed by the bound of the induction.
This is a straightforward generalization of how intersections of dimension at most two were handled in the proof of Theorem \ref{th:C3}.

When $k\ge d$ and $k=d'-1$, we rely on Lemma \ref{le:compFlats} to handle incidences between regular points $U$ and regular points of varieties of $\vars_k$.
By Lemma \ref{le:compFlats}, for each regular point $p\in U$ there exists a complex line $\ell_p \subset \CC^d$ such that $\ell_p \subset h$ for every hyperplane $h\in \vars$ satisfying $\dim (\phi(h)\cap U)= k$ and that $p$ is a regular point of $\phi(h)\cap U$.
We associate such a complex line with each regular $p\in U \cap \pts$, and denote the resulting set of lines as $\lines_U$. Let $g$ be a generic hyperplane in $\CC^d$, such that every line of $\lines_U$ intersects $g$ at a single point, and every hyperplane of $\vars$ intersects $g$ at $(d-2)$-flat.
Inside of $g$, the number of line-flat containments becomes the number of incidences between points and hyperplanes in $\CC^{d-1}$. (As explained in the proof of Theorem \ref{th:C3}, we need to separately handle lines that correspond to at least $s$ points of $U$.)
This number of incidences is bounded by the first induction hypothesis, and is subsumed by the bound of the induction step.
This is a straightforward generalization of the use of Lemma \ref{le:compFlatsC3} in the proof of Theorem \ref{th:C3}.

In the preceding paragraph we only bounded a specific type of incidences in the case of $k\ge d$ and $k=d'-1$.
We now consider the remaining incidences in this case.
By Theorem \ref{th:Singular}, the set of singular points of $U$ is a variety of dimension at most $d'-1$.
We separately study each of the irreducible components of the singular set, in the same way we separately study each component of the partition $\vb(f)$.
By Lemma \ref{le:BoundedNumComponents} there are $O_{d',k}(1)$ such irreducible components, each of dimension at most $d'-1$.
We also need to handle incidences between regular points of $U$ and singular points of the varieties of $\vars_k$.
By Theorem \ref{th:Singular}, the set of singular points of a $k$-dimensional variety $h\in \vars_k$ is a variety of dimension at most $k-1$.
We can handle these lower-dimensional varieties in the same step we handle lower-dimensional intersections with $U$.
Once again, this is a straightforward generalization of the corresponding part of the proof of Theorem \ref{th:C3}.

It remains to consider the case where $k\ge d$ and $k<d'-1$.
In this case we have $(d-1)/d > k/d'$.
We may thus bound the number of incidences as in the proof of Theorem \ref{th:IncReal}.
By Lemma \ref{le:ratio test}, the resulting incidence bound would be subsumed in the bound of the induction step.
While repeating the proof of Theorem \ref{th:IncReal}, we obtain various subproblems involving incidences between $k''$-dimensional varieties and points on a $d''$-dimensional variety.
For each subproblem we repeat the above: If $k''<d$ we project onto a generic $d$-flat, when $k''\ge d$ and $k''=d''-1$ we rely on Lemma \ref{le:compFlats}, and otherwise we repeat the analysis of Theorem \ref{th:IncReal}.

There are several recursive steps in the above analysis, where at each step we decrease either the dimension of $U$ or the dimension of the varieties we are working with.
It is not difficult to verify that the number of steps is $O_{d,k}(1)$ and the bound obtained by each step is subsumed by the bound of the induction step.
We close the induction step just as in the proof of Theorem \ref{th:C3}, by taking sufficiently large constants $\alpha_1$ and $\alpha_2$.
Unlike the proof of Theorem \ref{th:C3}, these two constants depend on $d$ and $k$ such that $\alpha_{1,d,k}$ is sufficiently large with respect to $\alpha_{1,d-1,k}$ and $\alpha_{1,d,k-1}$, and similarly for $\alpha_{2,d,k}$.
\end{proof}


\end{document}